\documentclass[reqno, 12pt]{amsart}
\pdfoutput=1
\makeatletter
\let\origsection=\section \def\section{\@ifstar{\origsection*}{\mysection}}
\def\mysection{\@startsection{section}{1}\z@{.7\linespacing\@plus\linespacing}{.5\linespacing}{\normalfont\scshape\centering\S}}
\makeatother

\usepackage{amsmath,amssymb,amsthm}
\usepackage{mathrsfs}
\usepackage{mathabx}\changenotsign
\usepackage{dsfont}

\usepackage{xcolor}

\usepackage[backref]{hyperref}
\hypersetup{
    colorlinks,
    linkcolor={red!60!black},
    citecolor={green!60!black},
    urlcolor={blue!60!black}
}

\usepackage[open,openlevel=2,atend]{bookmark}

\usepackage[abbrev,msc-links,backrefs]{amsrefs}
\usepackage{doi}

\renewcommand{\PrintDOI}[1]{\doi{#1}}

\usepackage[T1]{fontenc}
\usepackage{lmodern}
\usepackage[babel]{microtype}
\usepackage[english]{babel}

\usepackage{geometry}
\numberwithin{equation}{section}
\numberwithin{figure}{section}

\usepackage{enumitem}
\def\rmlabel{\upshape({\itshape \roman*\,})}

\def\alabel{\upshape({\itshape \alph*\,})}

\def\nlabel{\upshape({\itshape \arabic*\,})}

\let\polishlcross=\l
\def\l{\ifmmode\ell\else\polishlcross\fi}

\def\qand{\quad\text{and}\quad}
\def\qqand{\qquad\text{and}\qquad}

\let\emptyset=\varnothing
\let\setminus=\smallsetminus

\let\sm=\setminus

\makeatletter
\def\moverlay{\mathpalette\mov@rlay}
\def\mov@rlay#1#2{\leavevmode\vtop{   \baselineskip\z@skip \lineskiplimit-\maxdimen
   \ialign{\hfil$\m@th#1##$\hfil\cr#2\crcr}}}
\newcommand{\charfusion}[3][\mathord]{
    #1{\ifx#1\mathop\vphantom{#2}\fi
        \mathpalette\mov@rlay{#2\cr#3}
      }
    \ifx#1\mathop\expandafter\displaylimits\fi}
\makeatother

\newcommand{\dcup}{\charfusion[\mathbin]{\cup}{\cdot}}
\newcommand{\bigdcup}{\charfusion[\mathop]{\bigcup}{\cdot}}

\DeclareFontFamily{U}  {MnSymbolC}{}
\DeclareSymbolFont{MnSyC}         {U}  {MnSymbolC}{m}{n}
\DeclareFontShape{U}{MnSymbolC}{m}{n}{
    <-6>  MnSymbolC5
   <6-7>  MnSymbolC6
   <7-8>  MnSymbolC7
   <8-9>  MnSymbolC8
   <9-10> MnSymbolC9
  <10-12> MnSymbolC10
  <12->   MnSymbolC12}{}
\DeclareMathSymbol{\powerset}{\mathord}{MnSyC}{180}

\usepackage{tikz}
\usetikzlibrary{calc,decorations.pathmorphing}
\pgfdeclarelayer{background}
\pgfdeclarelayer{foreground}
\pgfdeclarelayer{front}
\pgfsetlayers{background,main,foreground,front}

\newcommand{\qedge}[7]{

	\ifx\relax#4\relax
		\def\qoffs{0pt}
	\else
		\def\qoffs{#4}
	\fi

	\def\qhedge{
		($#1+#3!\qoffs!-90:#2-#3$) --
		($#2+#1!\qoffs!-90:#3-#1$) --
		($#3+#2!\qoffs!-90:#1-#2$) -- cycle}

	\coordinate (12) at ($#1!\qoffs!90:#2$);
	\coordinate (13) at ($#1!\qoffs!-90:#3$);
	\coordinate (23) at ($#2!\qoffs!90:#3$);
	\coordinate (21) at ($#2!\qoffs!-90:#1$);
	\coordinate (31) at ($#3!\qoffs!90:#1$);
	\coordinate (32) at ($#3!\qoffs!-90:#2$);
	
	\def\nqhedge{
		(13) let \p1=($(13)-#1$), \p2=($(12)-#1$) in
			arc[start angle={atan2(\y1,\x1)}, delta angle={atan2(\y2,\x2)-atan2(\y1,\x1)-360*(atan2(\y2,\x2)-atan2(\y1,\x1)>0)}, x radius=\qoffs, y radius=\qoffs] --
		(21) let \p1=($(21)-#2$), \p2=($(23)-#2$) in
			arc[start angle={atan2(\y1,\x1)}, delta angle={atan2(\y2,\x2)-atan2(\y1,\x1)-360*(atan2(\y2,\x2)-atan2(\y1,\x1)>0)}, x radius=\qoffs, y radius=\qoffs] --
		(32) let \p1=($(32)-#3$), \p2=($(31)-#3$) in
			arc[start angle={atan2(\y1,\x1)}, delta angle={atan2(\y2,\x2)-atan2(\y1,\x1)-360*(atan2(\y2,\x2)-atan2(\y1,\x1)>0)}, x radius=\qoffs, y radius=\qoffs] --
		cycle}

		\ifx\relax#5\relax
		\def\qlwidth{1pt}
	\else
		\def\qlwidth{#5}
	\fi
	
		\ifx\relax#7\relax
		\fill \nqhedge;
	\else
		\fill[#7]\nqhedge;
	\fi

		\ifx\relax#6\relax
		\draw[line width=\qlwidth,rounded corners=\qoffs]\nqhedge;
	\else
		\draw[line width=\qlwidth,#6]\nqhedge;
	\fi
}

\let\epsilon=\varepsilon
\let\eps=\epsilon
\let\rho=\varrho
\let\theta=\vartheta

\let\phi=\varphi

\def\NN{{\mathds N}}

\def\RR{{\mathds R}}

\newcommand{\cA}{\mathcal{A}}

\newcommand{\cK}{\mathcal{K}}

\newcommand{\cP}{\mathcal{P}}

\newcommand{\ccA}{\mathscr{A}}

\newcommand{\ccP}{\mathscr{P}}

\theoremstyle{plain}
\newtheorem{thm}{Theorem}[section]
\newtheorem{fact}[thm]{Fact}
\newtheorem{prop}[thm]{Proposition}

\newtheorem{cor}[thm]{Corollary}
\newtheorem{lemma}[thm]{Lemma}

\theoremstyle{definition}

\newtheorem{dfn}[thm]{Definition}

\newtheorem{question}[thm]{Question}
\newtheorem{problem}[thm]{Problem}

\usepackage{accents}
\newcommand{\seq}[1]{\accentset{\rightharpoonup}{#1}}
\def\colond{\colon\,}
\let\st=\colond
\def\ex{\mathrm{ex}}
\def\Hom{\mathrm{Hom}}
\def\hom{\mathrm{hom}}
\def\rmd{\mathrm{d}}

\usepackage{scalerel}
\newsavebox\vdegbox
\savebox\vdegbox{\tikz{
		\draw[black,fill=black] (90:1) circle (.35);
		\draw[black,line width=0.10cm] (210:1) circle (.30);
		\draw[black,line width=0.10cm] (330:1) circle (.30);
		\draw[opacity=0] (0:1.2) circle (0.1);
	}}

\newsavebox\vvbox
\savebox\vvbox{\tikz{
		\draw[black,line width=0.10cm] (90:1) circle (.30);
		\draw[black,fill=black] (210:1) circle (.35);
		\draw[black,fill=black] (330:1) circle (.35);
		\draw[opacity=0] (0:1.2) circle (0.1);
	}}

\newsavebox\pdegbox
\savebox\pdegbox{\tikz{
		\draw[black,line width=0.10cm] (90:1) circle (.30);
		\draw[black,fill=black] (210:1) circle (.35);
		\draw[black,fill=black] (330:1) circle (.35);
		\draw[black,line width=0.28cm ] (210:1) -- (330:1);
		\draw[opacity=0] (0:1.2) circle (0.1);
	}}

\newsavebox\vvvbox
\savebox\vvvbox{\tikz{
		\draw[black,fill=black] (90:1) circle (.35);
		\draw[black,fill=black] (210:1) circle (.35);
		\draw[black,fill=black] (330:1) circle (.35);
		\draw[opacity=0] (0:1.2) circle (0.1);
	}}
\newcommand{\vvv}{\mathord{\scaleobj{1.2}{\scalerel*{\usebox{\vvvbox}}{x}}}}
\newcommand{\pivvv}{\pi_{\vvv}}

\newsavebox\evbox
\savebox\evbox{\tikz{
		\draw[black,fill=black] (90:1) circle (.35);
		\draw[black,fill=black] (210:1) circle (.35);
		\draw[black,fill=black] (330:1) circle (.35);
		\draw[black,line width=0.28cm ] (210:1) -- (330:1);
		\draw[opacity=0] (0:1.2) circle (0.1);
	}}
\newcommand{\ev}{\mathord{\scaleobj{1.2}{\scalerel*{\usebox{\evbox}}{x}}}}

\newsavebox\eebox
\savebox\eebox{\tikz{
		\draw[black,fill=black] (90:1) circle (.35);
		\draw[black,fill=black] (210:1) circle (.35);
		\draw[black,fill=black] (330:1) circle (.35);
		\draw[black,line width=0.28cm ] (90:1) -- (330:1);
		\draw[black,line width=0.28cm ] (90:1) -- (210:1);
		\draw[opacity=0] (0:1.2) circle (0.1);
	}}
\newcommand{\ee}{\mathord{\scaleobj{1.2}{\scalerel*{\usebox{\eebox}}{x}}}}
\newcommand{\piee}{\pi_{\ee}}

\newsavebox\eeebox
\savebox\eeebox{\tikz{
		\draw[black,fill=black] (90:1) circle (.35);
		\draw[black,fill=black] (210:1) circle (.35);
		\draw[black,fill=black] (330:1) circle (.35);
		\draw[black,line width=0.28cm ] (90:1) -- (330:1);
		\draw[black,line width=0.28cm ] (90:1) -- (210:1);
		\draw[black,line width=0.28cm ] (210:1) -- (330:1);
		\draw[opacity=0] (0:1.2) circle (0.1);
	}}

\def\red{{\rm red}}
\def\blue{{\rm blue}}
\def\green{{\rm green}}

\begin{document}

\title[Hypergraphs with vanishing Tur\'an density]
{Hypergraphs with vanishing Tur\'an density in uniformly dense hypergraphs}

\author[Chr.~Reiher]{Christian Reiher}
\address{Fachbereich Mathematik, Universit\"at Hamburg, Hamburg, Germany}
\email{Christian.Reiher@uni-hamburg.de}
\email{schacht@math.uni-hamburg.de}

\author[V.~R\"{o}dl]{Vojt\v{e}ch R\"{o}dl}
\address{Department of Mathematics and Computer Science,
Emory University, Atlanta, USA}
\email{rodl@mathcs.emory.edu}
\thanks{The second author is supported by NSF grant DMS 1301698.}

\author[M.~Schacht]{Mathias Schacht}


\begin{abstract}
P.~Erd\H{o}s [\emph{On extremal problems of graphs and generalized graphs},
Israel Journal of Mathematics~{\bf 2} (1964), 183--190] characterised those
hypergraphs $F$ that have to appear in any sufficiently large hypergraph $H$ of positive
density. We study related questions for $3$-uniform hypergraphs with the additional 
assumption that $H$ has to be {\it uniformly dense} with respect to vertex sets.
In particular, we characterise those hypergraphs $F$ that are guaranteed to
appear in large uniformly dense hypergraphs~$H$ of positive density. We also 
review the case when the density of the induced subhypergraphs of $H$
may depend on the proportion of the considered vertex sets.
\end{abstract}

\maketitle

\section{Introduction}
\label{sec:intro}
Unless said otherwise, all hypergraphs considered here are $3$-uniform. For such a 
hypergraph $H=(V,E)$ the set of vertices is 
denoted by $V=V(H)$ and we refer to the set of hyperedges by $E=E(H)$.  
Moreover, we denote by $\partial H\subseteq V^{(2)}$ the subset of all two~element 
subsets of~$V$, that contains all pairs covered by some hyperedge $e\in E$. For a
hyperedge $\{x,y,z\}\in E$ we sometimes simply write $xyz\in E$.

A classical extremal problem introduced by Tur\'an~\cite{Tu41} asks to study for
a given hypergraph $F$ its extremal function $\ex(n, F)$ sending
each positive integer to the maximum number of edges that a hypergraph of order $n$
can have without containing $F$ as a subhypergraph. In particular, one often focuses  
on the {\it Tur\'an density} $\pi(F)$ of~$F$ defined by
\[
	\pi(F)=\lim_{n\to\infty}\frac{\ex(n, F)}{\binom{n}{3}}\,.
\]
The problem to determine the Tur\'an densities of all hypergraphs is known to be very 
hard and so far it has been solved for a few hypergraphs only. A general result in this 
area due to Erd\H{o}s~\cite{Er64} asserts that a hypergraph $F$ satisfies $\pi(F)=0$
if and only if it is tripartite in the sense that there is a partition $V(F)=X\dcup Y\dcup Z$
such that every edge of $F$ contains precisely one vertex from each of $X$, $Y$, and $Z$.
  
Following a suggestion by Erd\H{o}s and S\'os~\cite{ErSo82} we studied variants
of Tur\'an's problem for {\it uniformly dense} hypergraphs~\cites{RRS-a, RRS-b, RRS-d, RRS-e}.  
Instead of finding the desired hypergraph~$F$ in an arbitrary ``host'' hypergraph $H$ of 
sufficiently large density one assumes in these problems that there are no ``sparse spots''
in the edge distribution of $H$.
There are various ways
to make this precise and we refer to~\cite{RRS-b}*{Section~4} and~\cite{RRS-e}*{Section~2}
for a more detailed discussion. Here we consider two closely related concepts, where the hereditary density condition pertains to large sets of vertices (see Sections~\ref{sec:udpd} 
and~\ref{sec:udvd} below).

\subsection{Uniformly dense hypergraphs with positive density}
\label{sec:udpd}
The first concept we discuss here continues our work from~\cites{RRS-a, RRS-b, RRS-d, RRS-e}.
Roughly speaking, this notion guarantees density~$d$ for all hypergraphs induced on 
sufficiently large vertex sets of linear size.
\begin{dfn} \label{dfn:vtxdense}
For real numbers $d\in[0, 1]$ and $\eta>0$ we say that a hypergraph 
$H=(V, E)$ is {\it $(d, \eta, 1)$-dense} if for all $U\subseteq V$ the 
estimate 
\[
	\big|U^{(3)}\cap E\big|\ge d\binom{|U|}{3}-\eta\,|V|^3
\]
holds, where  $U^{(3)}$ denotes the set of all three element subsets of $U$.
\end{dfn} 
The Tur\'an densities associated with this concept are defined by 
\begin{multline*}
	\pi_1(F)=\sup\bigl\{d\in[0,1]\colon \text{for every $\eta>0$ and $n\in \NN$ there exists}\\
	\text{an $F$-free,  $(d,\eta, 1)$-dense hypergraph $H$ with $|V(H)|\geq n$}\bigr\}\,.
\end{multline*}

Our main result characterises all hypergraphs $F$ with $\pi_1(F)=0$.
\begin{thm}\label{zero}
For a $3$-uniform 
hypergraph $F$, the following are equivalent:
\begin{enumerate}[label=\alabel]
	\item \label{zero:a}$\pi_1(F)=0$.
	\item \label{zero:b} There is an enumeration of the vertex set $V(F)=\{v_1, \dots, v_f\}$ 
		and there is a three-colouring $\phi\colon \partial F\to\{\red,\blue,\green\}$ of the pairs of vertices covered by hyperedges of~$F$ 
		such that every hyperedge $\{v_i,v_j,v_k\}\in E(F)$ with $i<j<k$ satisfies 
		\[
			\phi(v_i,v_j)=\red,\quad \phi(v_i,v_k)=\blue, \qand \phi(v_j,v_k)=\green.
		\]
\end{enumerate}
\end{thm}

It is easy to see that tripartite hypergraphs~$F$ satisfy condition~\ref{zero:b}.
Moreover, it follows from the work in~\cite{KNRS} that 
every linear hypergraph~$F$ satisfies $\pi_1(F)=0$. Linear hypergraphs have the property that every element of 
$\partial F$ is contained in precisely one hyperedge of $F$. Consequently, 
we may consider an arbitrary vertex enumeration of~$F$ and then a colouring of $\partial F$
satisfying condition~\ref{zero:b} is forced. 
However, there are hypergraphs 
displaying condition~\ref{zero:b}, that are neither tripartite nor linear. For example, 
one can check that the hypergraph obtained from the tight cycle on five vertices by removing 
one hyperedge is such a hypergraph~$F$ (see~Figure~\ref{fig:C5minus}).

\begin{figure}[ht]\label{fig:C5minus}
\centering
\begin{tikzpicture}[scale=0.8]
	
	\coordinate (r1) at (0, 4);
	\coordinate (r2) at (1.17, 1.17);
	\coordinate (r3) at (4, 0);
	\coordinate (r4) at (6.83, 1.17);
	\coordinate (r5) at (8, 4);		
			
	\draw[red, line width=2pt] (r2) -- (r1) -- (r4) -- (r3);
	\draw[blue, line width=2pt] (r1) -- (r3) -- (r5) -- (r1);
	\draw[green, line width=2pt] (r3) -- (r2) -- (r5) -- (r4);
	
	\foreach \i in {r1, r2, r3, r4, r5}
			\fill  (\i) circle (2pt);
			
	\qedge{(r1)}{(r3)}{(r2)}{4.5pt}{1.5pt}{red!70!black}{red!70!black,opacity=0.2};
	\qedge{(r3)}{(r5)}{(r4)}{4.5pt}{1.5pt}{red!70!black}{red!70!black,opacity=0.2};
	\qedge{(r4)}{(r1)}{(r5)}{4.5pt}{1.5pt}{red!70!black}{red!70!black,opacity=0.2};	
	\qedge{(r5)}{(r2)}{(r1)}{4.5pt}{1.5pt}{red!70!black}{red!70!black,opacity=0.2};		
	
	\node at (-0.7,4) {$x$};
	\node at (0.7,0.7) {$w$};
	\node at (4,-0.7) {$v$};
	\node at (7.3,0.7) {$z$};
	\node at (8.7,4) {$y$};		
\end{tikzpicture}
\caption{Colouring of $\partial C^{(3)-}_5$ showing that $\pi_{1}\bigl(C^{(3)-}_5\bigr)=0$.
The ordering demanded by Theorem~\ref{zero}~\ref{zero:b} is from left to right, i.e.,
$x<w<v<z<y$, whereas on the cycle the vertices are ordered alphabetically with edges 
$vwx, wxy, xyz, yzv$.}
\label{fig:1}
\end{figure}
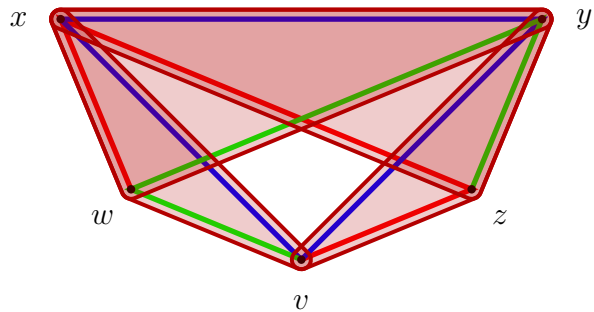

The easier implication of Theorem~\ref{zero} is ``\ref{zero:a} $\Longrightarrow$~\ref{zero:b}.''
For its proof we exhibit a ``universal'' hypergraph $H$ all of whose subhypergraphs obey 
condition~\ref{zero:b} and all of whose linear sized induced subhypergraphs have density 
$\tfrac 1{27}-o(1)$. In other words, our argument establishing this implication does
actually yield the following strengthening.

\begin{fact}\label{fact:jump}
If a hypergraph $F$ does not have property~\ref{zero:b} from Theorem~\ref{zero}, then 
$\pi_1(F)\ge\tfrac 1{27}$.
\end{fact}

\begin{proof}
Given a positive integer $n$ consider a three-colouring 
$\phi\colon [n]^{(2)}\to\{\red,\blue,\green\}$ 
of the pairs of the first $n$ positive integers. We define a hypergraph $H_{\phi}$
with vertex set $[n]$ by regarding a triple $\{i,j,k\}$ with $1\leq i<j<k\leq n$ 
as being a hyperedge if and only if $\phi(i,j)=\red$, $\phi(i,k)=\blue$, 
and $\phi(j,k)=\green$. Standard probabilistic arguments show that when
$\phi$ is chosen uniformly at random, then for any fixed $\eta>0$ the probability 
that~$H_{\phi}$ is $(1/27,\eta, 1)$-dense tends to $1$ as $n$ tends to infinity.
On the other hand, as~$F$ does not satisfy condition~\ref{zero:b} from Theorem~\ref{zero}, 
it is in a deterministic sense the case that $F$ is never a subgraph of $H_\phi$ no matter
how large $n$ becomes. Thus we have indeed~$\pi_{1}(F)\ge\tfrac 1{27}$.
\end{proof}

The combination of Theorem~\ref{zero} and Fact~\ref{fact:jump} leads
immediately to the following consequence, which shows that $\pi_1$ ``jumps'' 
from~$0$ to at least $\tfrac 1{27}$. 

\begin{cor}\label{jump}
If a hypergraph $F$ satisfies $\pi_1{(F)}>0$, then $\pi_1{(F)}\ge\tfrac1{27}$.
\end{cor}

At this point the optimality of Corollary~\ref{jump} is unknown and 
it remains an open problem to determine the infimum over all non-zero values of $\pi_1(\cdot)$.

\subsection{Uniformly dense hypergraphs with vanishing density}
\label{sec:udvd}
The second concept we discuss here is closely related to the one from Definition~\ref{dfn:vtxdense}.
It was introduced by Erd\H{o}s and S\'os in~\cite{ErSo82} (see also~\cite{Er90}*{page~24}). To prepare its definition we need 
a concept of being $d$-dense when $d$ can be a function rather than just a single number
and we shall consider {\it sequences} of hypergraphs
instead of just one individual hypergraph.  

\begin{dfn}\phantomsection 
	\label{dfn:d-dense}
\begin{enumerate}[label=\alabel]
\item\label{it:d-dense} Let $\seq{H}=(H_n)_{n\in\NN}$ be a sequence of hypergraphs with 
$|V(H_n)|\to\infty$ as $n \to\infty$ and let
$d\colon (0, 1)\longrightarrow (0, 1)$ be a function. We say that
$\seq{H}$ is {\it $d$-dense} provided that for every $\eta\in (0, 1)$
there is an $n_0\in\NN$ such that for $n\ge n_0$ every $U\subseteq V(H_n)$
with $|U|\ge \eta\,|V(H_n)|$ satisfies
\[
	\big|U^{(3)}\cap E(H_n)\big|\geq d(\eta)\binom{|U|}{3}\,.
\] 
\item A hypergraph $F$ is called {\it frequent} if for every function 
$d\colon (0, 1)\longrightarrow (0, 1)$ and every $d$-dense sequence 
$\seq{H}=(H_n)_{n\in\NN}$ of hypergraphs there is an integer $n_0$
such that $F$ is a subhypergraph of every $H_n$ with $n\ge n_0$.
\end{enumerate}
\end{dfn}

Erd\H{o}s and S\'os~\cite{ErSo82}*{Proposition~3} described the following instructive example $(T_n)_{n\in\NN}$ 
of a sequence of \emph{ternary hypergraphs} that is $d$-dense for some function $d(\cdot)$, but not uniformly dense in the sense of Definition~\ref{dfn:vtxdense}. 
Take the vertex 
set of $T_n$ to be the set~$\{0, 1, 2\}^n$ of all sequences with length $n$ all of whose 
entries are $0$, $1$, or $2$. 
Given three distinct vertices of $T_n$, say $\seq{x}=(x_1, \dots, x_n)$,
$\seq{y}=(y_1, \dots, y_n)$, and $\seq{z}=(z_1, \dots, z_n)$ there is a least integer~$i\in [n]$ for which $x_i=y_i=z_i$ is not the case and we put a hyperedge~$\{\seq{x}, \seq{y}, \seq{z}\}$ into~$E(T_n)$ if and only if this index $i$ 
satisfies $\{x_i, y_i, z_i\}=\{0, 1, 2\}$. 
It was stated in~\cite{ErSo82} that the sequence of ternary hypergraphs is $d$-dense for 
some appropriate function $d(\cdot)$ and a short proof of this fact appeared in~\cite{FR88}.
In Section~\ref{sec:Erdoes} we obtain the following improvement.
\begin{prop}\label{prop:toptimal}
	The sequence of ternary hypergraphs $(T_n)_{n\in\NN}$ is $d$-dense for any 
	function $d\colon(0,1]\to (0,1]$ with 
	$d(\eta)<\tfrac{1}{4}\eta^{\frac{2}{\log_2(3)-1}}$.
\end{prop}
Considering subsets $U\subseteq V(T_n)$ of the form $U=\{0,1\}^r\times \{0,1,2\}^{n-r}$ shows that 
Proposition~\ref{prop:toptimal} is optimal whenever $\eta=(2/3)^r$ 
for some $r\in\NN$.
Since ternary hypergraphs are $d$-dense for some function $d(\cdot)$, it follows 
that every frequent hypergraph must be
contained in some ternary hypergraph and Erd\H{o}s wondered in~\cite{Er90}
whether the converse of this holds as well. This was indeed verified 
by Frankl and R\"odl in~\cite{FR88} and the following characterisation can be viewed as an 
analogue of Theorem~\ref{zero} for $d$-dense hypergraphs.

\begin{thm}\label{thm:ternary}
A hypergraph $F$ is frequent if, and only if it occurs as a subhypergraph of a
ternary hypergraph. \qed
\end{thm}

It is not hard to show (see Lemma~\ref{lem:decidable}) that if $F$ is a subhypergraph of some ternary hypergraph, then~$F\subseteq T_{|V(F)|}$ and, consequently, Theorem~\ref{thm:ternary} entails,
that it is decidable whether a given hypergraph is frequent or not.

\subsection*{Organisation} The proof of the implication 
``\ref{zero:b} $\Longrightarrow$~\ref{zero:a}'' of Theorem~\ref{zero} utilises the hypergraph 
regularity method that is revisited in Section~\ref{sec:regular}. This method allows us in 
Section~\ref{sec:reduce} to reduce the problem of embedding hypergraphs satisfying the 
condition~\ref{zero:b} in Theorem~\ref{zero} into uniformly dense hypergraphs to a problem
concerning so-called {\it reduced hypergraphs}. This reduction will be carried out in 
Section~\ref{sec:reduce} and the main argument will then be given in Section~\ref{sec:proof}.
In Section~\ref{sec:Erdoes} we prove Proposition~\ref{prop:toptimal}, which implies the forward implication of Theorem~\ref{thm:ternary}. 

For a more complete presentation we include a short proof of the backward implication of Theorem~\ref{thm:ternary} as well, which follows the lines of the proof in~\cite{FR88}. 
In contrast to the proof of the implication ``\ref{zero:b} $\Longrightarrow$~\ref{zero:a}'' of Theorem~\ref{zero}  this proof is somewhat simpler
and is based on a supersaturation argument. Extensions of our results to $k$-uniform hypergraphs with $k>3$
will be discussed in the concluding remarks.

\section{Hypergraph regularity}
\label{sec:regular}

A key tool in the proof of Theorem~\ref{zero} is the regularity lemma for $3$-uniform hypergraphs. 
We follow the approach from~\cites{RoSchRL,RoSchCL} combined with the results from~\cite{Gow06} 
and~\cite{NPRS09}. 

For two disjoint sets $X$ and $Y$ we denote by $K(X,Y)$ the complete bipartite graph with 
that vertex partition.
We say that a bipartite graph $P=(X\dcup Y,E)$ is \emph{$(\delta_2, d_2)$-regular} 
if for all subsets 
$X'\subseteq X$ and $Y'\subseteq Y$ we have 
\[
	\big|e(X',Y')-d_2|X'||Y'|\big|\leq \delta_2 |X||Y|\,,
\]
where $e(X',Y')$ denotes the number of edges of $P$ with one vertex in $X'$ and one vertex in~$Y'$.
Moreover, for $k\geq 2$ we say a $k$-partite graph $P=(X_1\dcup \dots\dcup X_k,E)$ is $(\delta_2, d_2)$-regular, 
if all its  $\binom{k}{2}$ naturally 
induced bipartite subgraphs $P[X_i,X_j]$ are $(\delta_2, d_2)$-regular. 
For a tripartite graph $P=(X\dcup Y\dcup Z,E)$
we denote by $\cK_3(P)$ the triples of vertices spanning a triangle in~$P$, i.e., 
\[
	\cK_3(P)=\big\{\{x,y,z\}\subseteq X\cup Y\cup Z\colon xy, xz, yz\in E\big\}\,.
\]
If the tripartite graph $P$ is $(\delta_2, d_2)$-regular, then the \emph{triangle counting lemma}
implies
\begin{equation}
	\label{eq:TCL}
		|\cK_3(P)|\leq d_2^3|X||Y||Z|+3\delta_2|X||Y||Z|\,.
\end{equation}

We say a $3$-uniform hypergraph $H=(V,E_H)$ is regular w.r.t.\ a tripartite graph $P$ if it matches 
approximately
the same proportion of triangles for every subgraph $Q\subseteq P$. 

\begin{dfn}
\label{def:reg}
A $3$-uniform hypergraph $H=(V,E_H)$ is \emph{$(\delta_3,d_3)$-regular w.r.t.\ 
a tripartite graph $P=(X\dcup Y\dcup Z,E_P)$} 
with $V\supseteq  X\cup Y\cup Z$ if for every tripartite subgraph $Q\subseteq P$ we have 
\[
	\big||E_H\cap\cK_3(Q)|-d_3|\cK_3(Q)|\big|\leq \delta_3|\cK_3(P)|\,.
\]
Moreover, we simply say \emph{$H$ is $\delta_3$-regular w.r.t.\ $P$}, if it is $(\delta_3,d_3)$-regular for some $d_3\geq 0$.
We also define the \emph{relative density} of $H$ w.r.t.\ $P$ by
\[
	d(H|P)=\frac{|E_H\cap\cK_3(P)|}{|\cK_3(P)|}\,,
\]
where we use the convention $d(H|P)=0$ if $\cK_3(P)=\emptyset$. If $H$ is not $\delta_3$-regular w.r.t.\ $P$, then we simply refer to it as \emph{$\delta_3$-irregular}.
\end{dfn}

The regularity lemma for $3$-uniform hypergraphs, introduced by Frankl and R\"odl in~\cite{FR}, provides for 
a hypergraph $H$ a partition of its vertex set and a partition of the edge sets of the complete bipartite 
graphs induced by the vertex partition such that for appropriate constants $\delta_3$, $\delta_2$,~and~$d_2$ 
\begin{enumerate}[label=\nlabel]
	\item the bipartite graphs given by the partitions are $(\delta_2,d_2)$-regular and
	\item $H$ is $\delta_3$-regular for ``most'' tripartite graphs $P$ given by the partition.
\end{enumerate}
In many proofs based on the regularity method it is
convenient to ``clean'' the regular partition provided by the lemma. In particular, 
we shall disregard hyperedges of~$H$ that belong to $\cK_3(P)$ where $H$ is not $\delta_3$-regular or 
where $d(H|P)$ is very small. These properties are rendered in the following somewhat standard
corollary of the regularity lemma.

\begin{thm}
	\label{thm:TuRL}\pushQED{\qed} 
	For every $d_3>0$, $\delta_3>0$ and  $m\in\NN$, and every function 
	$\delta_2\colond \NN \to (0,1]$,
	there exist integers~$T_0$ and $n_0$ such that for every $n\geq n_0$
	and every $n$-vertex $3$-uniform hypergraph $H=(V,E)$ the following holds.
	
	There exists a subhypergraph $\hat H=(\hat V,\hat E)\subseteq H$, an integer $\l\leq T_0$,
	a vertex partition $V_1\dcup\dots\dcup V_m=\hat V$, 
	and for all integers $i$, $j$ with $1\leq i<j\leq m$ there exists 
	a partition 
	$\cP^{ij}=\{P^{ij}_\alpha=(V_i\dcup V_j,E^{ij}_\alpha)\colond 1\leq \alpha \leq \l\}$ 
	of $K(V_i,V_j)$ satisfying the following properties
	\begin{enumerate}[label=\rmlabel]
		\item\label{TuRL:1} $|V_1|=\dots=|V_m|\geq (1-\delta_3)n/T_0$,
		\item for every $1\leq i<j\leq m$ and $\alpha\in [\l]$ the bipartite graph $P^{ij}_\alpha$ is $(\delta_2(\l),1/\l)$-regular,
		\item $\hat H$ is $\delta_3$-regular w.r.t. all tripartite graphs
		\begin{equation}\label{eq:triad}
			P^{ijk}_{\alpha\beta\gamma}
			=P^{ij}_\alpha\dcup P^{ik}_\beta\dcup P^{jk}_\gamma
			=(V_i\dcup V_j\dcup V_k, E^{ij}_\alpha\dcup E^{ik}_{\beta}\dcup E^{jk}_{\gamma})\,,
		\end{equation}
		with $1\leq i<j<k\leq m$ and $\alpha$, $\beta$, $\gamma\in[\l]$, 
		and $d({\hat H}|P^{ijk}_{\alpha\beta\gamma})$ is either $0$ or at least $d_3$, 
		\item\label{TuRL:4} and for every $1\leq i<j<k\leq m$ we have 
				\[e_{\hat H}(V_i,V_j,V_k)\geq e_{H}(V_i,V_j,V_k)-(d_3+\delta_3)|V_i||V_j||V_k|\,.\qedhere\]
	\end{enumerate}
	\popQED
\end{thm}
Owing to their special r\^ole we shall refer to the tripartite graphs considered in~\eqref{eq:triad} as \emph{triads}.

A proof of Theorem~\ref{thm:TuRL} based on a refined version of 
the regularity lemma from~\cite{RoSchRL}*{Theorem~2.3}
can be found in~\cite{RRS-a}*{Corollary 3.3}. 

We shall use the \emph{counting/embedding lemma}, which allows us to embed 
hypergraphs of fixed isomorphism type into appropriate and sufficiently regular and 
dense triads of the partition provided by Theorem~\ref{thm:TuRL}. 
It is a direct consequence of~\cite{NPRS09}*{Corollary~2.3}.

\begin{thm}[Embedding Lemma]
	\label{thm:EL}
	Let a hypergraph $F$ with vertex set $[f]$ and $d_3>0$ be given.
    Then there exist $\delta_3>0$ and functions 
	$\delta_2\colon \NN\to(0,1]$ and  $N\colon \NN\to\NN$
	such that the following holds for every $\l\in\NN$.
	 
	Suppose $P=(V_1\dcup\dots\dcup V_f, E_P)$ is a $(\delta_2(\l),\frac{1}{\l})$-regular, 
	$f$-partite graph whose vertex classes satisfy $|V_1|=\dots=|V_f|\geq N(\l)$ and 
	suppose $H$ is an $f$-partite, $3$-uniform hypergraph
	such that for all edges $ijk$ of $F$ we have
	\begin{enumerate}[label=\alabel]
		\item\label{EL:a} $H$ is $\delta_3$-regular w.r.t.\ to the tripartite graph 
		$P[V_i\dcup V_j\dcup V_k]$ and 
		\item\label{EL:b} $d(H|P[V_i\dcup V_j\dcup V_k])\geq d_3$,
	\end{enumerate} 
	then $H$ contains a copy of $F$. In fact, there is a monomorphism $q$ from $F$ to $H$
	with~$q(i)\in V_i$ for all $i\in[f]$.\qed
\end{thm}
In an application of Theorem~\ref{thm:EL} the tripartite graphs $P[V_i\dcup V_j\dcup V_k]$ 
in~\ref{EL:a} and~\ref{EL:b} will be given by triads $P^{ijk}_{\alpha\beta\gamma}$ from the 
partition given by Theorem~\ref{thm:TuRL}.
For the proof of the direction ``\ref{zero:b} $\Longrightarrow$~\ref{zero:a}'' 
of Theorem~\ref{zero} we consider for a fixed hypergraph $F$ obeying condition~\ref{zero:b}
and fixed $\eps>0$ a sufficiently large uniformly dense hypergraph~$H$ of density~$\eps$.
We will apply the regularity lemma in the form of Theorem~\ref{thm:TuRL} to~$H$. 
The main part of the proof concerns the appropriate selection of dense and regular triads, 
that are ready for an application of the embedding lemma. In Section~\ref{sec:reduce} we formulate
a statement about reduced hypergraphs telling us that such a selection is indeed possible
and in Section~\ref{sec:proof} we give its proof.

\section{Moving to reduced hypergraphs}
\label{sec:reduce}

In our intended application of the hypergraph regularity method we need to keep 
track which triads are dense and regular and natural structures for encoding such
information are so-called reduced hypergraphs. We follow the terminology introduced 
in~\cite{RRS-d}*{Section~3}. 

Consider any finite set of indices $I$, suppose that associated with any two distinct 
indices~$i, j\in I$ we have a finite nonempty set of vertices $\cP^{ij}$, and that for 
distinct pairs of indices the corresponding vertex classes are disjoint. Assume further 
that for any three distinct indices $i,j,k\in I$ we are given a tripartite 
hypergraph $\cA^{ijk}$ with vertex classes~$\cP^{ij}$,~$\cP^{ik}$, and $\cP^{jk}$. 
Under such circumstances we call the $\binom{|I|}{2}$-partite 
hypergraph $\cA$ defined by
\[
V(\cA)=\bigdcup_{\{i,j\}\in
I^{(2)}} \cP^{ij}
\qquad \text{ and } \qquad
E(\cA)=\bigdcup_{\{i,j, k\}\in
I^{(3)}} E(\cA^{ijk})
\]
a {\it reduced hypergraph}. We also refer to $I$ as the {\it index set} of $\cA$, to the sets 
$\cP^{ij}$ as the {\it vertex classes} of $\cA$, and to the hypergraphs $\cA^{ijk}$ as the 
{\it constituents} of $\cA$. The order of the indices appearing in the 
pairs and triples of the superscripts of the vertex classes
and constituents of $\cA$ plays no r\^ole here, i.e., $\cP^{ij}=\cP^{ji}$ and $\cA^{ijk}=\cA^{kij}$ etc.
For $\mu>0$ such a reduced hypergraph $\cA$ is said to 
be {\it $\mu$-dense} if 
\[
	|E(\cA^{ijk})|\ge \mu\,|\cP^{ij}|\,|\cP^{ik}|\,|\cP^{jk}|
\]
holds for every triple $\{i, j, k\}\in I^{(3)}$.

In the light of the hypergraph regularity method, the proof of Theorem~\ref{zero}
reduces to the following statement whose proof will be given in the next section.

\begin{lemma}\label{char-modif}
Given $\mu>0$ and $f\in\NN$ there exists an integer $m$ 
such that the following holds. 
If $\cA$ is a $\mu$-dense reduced hypergraph with index set $[m]$, vertex classes $\cP^{ij}$,
and constituents $\cA^{ijk}$,
then there are 
\begin{enumerate}[label=\rmlabel]
\item indices $\lambda(1)<\dots<\lambda(f)$ in $[m]$ and
\item for each pair $1\le r<s\le f$ there are three vertices 
$P^{\lambda(r)\lambda(s)}_{\red}$, $P^{\lambda(r)\lambda(s)}_{\blue}$, 
and $P^{\lambda(r)\lambda(s)}_{\green}$ in $\cP^{\lambda(r)\lambda(s)}$ 
\end{enumerate}
such that for every triple of indices $1\le r<s<t\le m$ the three vertices 
$P^{\lambda(r)\lambda(s)}_{\red}$, $P^{\lambda(r)\lambda(t)}_{\blue}$, 
and $P^{\lambda(s)\lambda(t)}_{\green}$ 
form a hyperedge in $\cA^{\lambda(r)\lambda(s)\lambda(t)}$.
\end{lemma}

At the end of this section we will prove that this lemma does indeed imply Theorem~\ref{zero}.
For this purpose it will be more convenient to work with an alternative definition of $\pi_1$
that we denote by $\pi_{\vvv}$.
In contrast to Definition~\ref{dfn:vtxdense}
it speaks about being dense with respect to three subsets of vertices rather than just one.  

\begin{dfn}
\label{pppqr}
A hypergraph $H=(V,E)$ of order $n=|V|$ is \emph{$(d,\eta,\vvv)$-dense} if 
for every triple of subsets $X, Y, 
Z\subseteq V$ the number $e_{\vvv}(X, Y, Z)$ of 
triples $(x,y,z)\in X\times Y\times Z$ 
with~$xyz\in E$ satisfies
\[
	e_{\vvv}(X, Y, Z)\ge d\,|X|\,|Y|\,|Z|- \eta\,n^3\,.
\] 
\end{dfn} 
Accordingly, we set
\begin{multline}\label{eq:pi3dot}
	\pi_{\vvv}(F)=\sup\bigl\{d\in[0,1]\colon \text{for every $\eta>0$ and $n\in \NN$ there exists}\\
	\text{an $F$-free, $(d,\eta, \vvv)$-dense hypergraph $H$ with $|V(H)|\geq n$}\bigr\}\,.
\end{multline}

Applying \cite{RRS-e}*{Proposition 2.5} to $k=3$ and $j=1$ we deduce that every 
hypergraph $F$ satisfies 
\begin{equation}\label{eq:3vs1}
	\pi_{\vvv}(F)=\pi_1(F)\,.
\end{equation}
Consequently it is allowed to imagine that in clause~\ref{zero:a}
of Theorem~\ref{zero} we would have written $\pi_{\vvv}(F)=0$ instead of $\pi_1(F)=0$.

\begin{proof}[Proof of Theorem~\ref{zero} assuming Lemma~\ref{char-modif}]
The implication ``\ref{zero:a} $\Longrightarrow$~\ref{zero:b}'' is implicit in Fact~\ref{fact:jump},
meaning that we just need to consider the reverse direction. Suppose to this end that a
hypergraph $F$ satisfying condition~\ref{zero:b} and some $\eps>0$ are given.
We need to check that for $\eps\gg \eta\gg n^{-1}$ every $(\eps, \eta, \vvv)$-dense
hypergraph $H$ of order $n$ contains a copy of~$F$. 

Of course, we may assume that $V(F)=[f]$ holds for some $f\in\NN$. 
Plugging $F$ and $d_3=\frac{\eps}{4}$ into the embedding lemma
we get a constant $\delta_3>0$, a function $\delta_2\colon \NN\to (0, 1]$, and a function
$N\colon \NN\to\NN$. Evidently we may assume that $\delta_3\le\frac{\eps}4$, that 
$\delta_2(\ell)\ll\ell^{-1}$, and that $N$ is increasing. Applying Lemma~\ref{char-modif}
with $\mu=\frac{\eps}{8}$ and $f$ we obtain an integer $m$. Given $d_3$, $\delta_3$, $m$,
and $\delta_2(\cdot)$ we get integers $T_0$ and $n_0$ from Theorem~\ref{thm:TuRL}. 
Finally we choose 
\[
	\eta=\frac{\eps(1-\delta_3)^3}{4T_0^3} \qqand n_1=2T_0\cdot N(T_0)\,.
\]

Now consider any $(\eps, \eta, \vvv)$-dense hypergraph $H$ of order $n\ge n_1$.
We contend that $F$ appears as a subhypergraph of $H$. To see this we take
\begin{enumerate}
\item[$\bullet$] a subhypergraph $\hat H=(\hat V,\hat E)\subseteq H$,
\item[$\bullet$] a vertex partition $V_1\dcup\dots\dcup V_m=\hat V$,
\item[$\bullet$] an integer $\ell\le T_0$,
\item[$\bullet$] and pair partitions 
$\cP^{ij}=\{P^{ij}_\alpha=(V_i\dcup V_j,E^{ij}_\alpha)\colond 1\leq \alpha \leq \l\}$ 
of $K(V_i, V_j)$ for all $1\leq i<j\leq m$
\end{enumerate}
satisfying the conditions~\ref{TuRL:1}--\ref{TuRL:4} from Theorem~\ref{thm:TuRL}.
The reduced hypergraph $\cA$ corresponding to this situation has index set $[m]$,
vertex classes $\cP^{ij}$ and a triple $\{P^{ij}_\alpha, P^{ik}_\beta, P^{jk}_\gamma\}$
is defined to be an edge of the constituent $\cA^{ijk}$ if and only if 
$d({\hat H}|P^{ijk}_{\alpha\beta\gamma})\ge d_3$. As we shall verify below, 
\begin{equation}\label{eq:Amudense}
	\cA \text{ is } \text{$\mu$-dense}.
\end{equation}
Due to Lemma~\ref{char-modif} this means that there are 
\begin{enumerate}
\item[$\bullet$] indices $\lambda(1)<\dots<\lambda(f)$ in $[m]$ and 
\item[$\bullet$] for each pair $1\le r<s\le f$ there are vertices
$P^{\lambda(r)\lambda(s)}_{\red}, P^{\lambda(r)\lambda(s)}_{\blue}, 
P^{\lambda(r)\lambda(s)}_{\green}\in \cP^{\lambda(r)\lambda(s)}$ 
\end{enumerate}
such that for every triple of indices $1\le r<s<t\le m$ the three vertices 
$P^{\lambda(r)\lambda(s)}_{\red}$, $P^{\lambda(r)\lambda(t)}_{\blue}$,
and $P^{\lambda(s)\lambda(t)}_{\green}$ 
form a hyperedge in $\cA^{\lambda(r)\lambda(s)\lambda(t)}$.
These vertices correspond to bipartite graphs forming dense regular 
triads. Since we have
\[
	|V_{\lambda(1)}|=\dots=|V_{\lambda(f)}|\ge \frac{(1-\delta_3)n}{T_0}
	\ge \frac{n_1}{2T_0}=N(T_0)\ge N(\ell)\,,
\]
the embedding lemma is applicable to the hypergraph $\hat H$ 
and to the $f$-partite graph with vertex partition
$\bigdcup_{r\in[f]}V_{\lambda(r)}$ and edge set 
$\bigdcup_{rs\in\partial F}P^{\lambda(r)\lambda(s)}_{\phi(\lambda(r), \lambda(s))}$,
where $\phi\colon \partial F\to\{\red,\blue,\green\}$ denotes any colouring
exemplifying that $F$ does indeed possess property~\ref{zero:b} from Theorem~\ref{zero}.
Consequently, the monomorphism guaranteed by Theorem~\ref{thm:EL} yields a copy of $F$ in $\hat H\subseteq H$.

So to conclude the proof it only remains to verify~\eqref{eq:Amudense}. Suppose to this end that 
some triple $\{i, j, k\}\in [m]^3$ is given. We have to verify that
\begin{equation}\label{eq:33}
	|E(\cA^{ijk})|\ge \mu\,|\cP^{ij}|\,|\cP^{ik}|\,|\cP^{jk}|=\frac{\eps}8\ell^3\,.
\end{equation}

Using that $H$ is $(\eps, \eta, \vvv)$-dense we infer
\[
	e_H(V_i, V_j, V_k)\ge \eps\,|V_i|\,|V_j|\,|V_k|-\eta n^3
\]
and by our choice of $\eta$ it follows that 
\[
	|V_i|\,|V_j|\,|V_k|\ge \left(\frac{(1-\delta_3)}{T_0}\right)^3n^3= \frac{4\eta}{\eps} n^3\,.
\]
So altogether we have 
\[
	e_H(V_i, V_j, V_k)\ge\tfrac 34\eps\,|V_i|\,|V_j|\,|V_k|\,.
\] 	
In combination with $\delta_3\le\frac\eps{4}=d_3$ and condition~\ref{TuRL:4} from
Theorem~\ref{thm:TuRL} this entails
\begin{equation}\label{eq:ehut}
	e_{\hat H}(V_i, V_j, V_k)\ge\tfrac 14\eps\,|V_i|\,|V_j|\,|V_k|\,.
\end{equation}

On the other hand, by the triangle counting lemma~\eqref{eq:TCL} and $\delta_2\ll\ell^{-1}$
each triad $P^{ijk}_{\alpha\beta\gamma}$ satisfies 
\[
	\cK_3\bigl(P^{ijk}_{\alpha\beta\gamma}\bigr)
	\le \bigl(\ell^{-3}+3\delta_2(\ell)\bigr)|V_i|\,|V_j|\,|V_k|
	\le 2\ell^{-3} |V_i|\,|V_j|\,|V_k|\,,
\]
for which reason 
\[
	e_{\hat H}(V_i, V_j, V_k)\le |E(\cA^{ijk})|\cdot 2\ell^{-3} |V_i|\,|V_j|\,|V_k|\,.
\]

Together with~\eqref{eq:ehut} this proves~\eqref{eq:33} and, hence, the implication from
Lemma~\ref{char-modif} to Theorem~\ref{zero}.
\end{proof}

\section{Proof of Theorem~\ref{zero}}
\label{sec:proof}

This entire section is devoted to the proof of Lemma~\ref{char-modif}. We begin by outlining 
the main ideas of this proof. 
The argument proceeds in three stages. In the first of them we will choose a subset 
$X\subseteq [m]$ and for any two indices $r<s$ from $X$ some vertex $P^{rs}_{\red}\in \cP^{rs}$ 
such that if $r<s<t$ are from~$X$, then $P^{rs}_{\red}$ has large degree in 
$\cA^{rst}$, where ``large'' means at least $\mu'\,|\cP^{rt}|\,|\cP^{st}|$ for some 
$\mu'$ depending only on $\mu$. This argument will have the property that for fixed 
$f$ and $\mu$ the size of~$X$ can be made as large as we wish by starting from a sufficiently
large $m$. 
Then, in the next stage, we shrink the set $X$ further to some 
$Y\subseteq X$ and select vertices $P^{rt}_{\blue}\in \cP^{rt}$ for all indices $r<t$ 
from $Y$ such that if $r<s<t$ are from~$Y$ then the pair-degree of $P^{rs}_{\red}$ and 
$P^{rt}_{\blue}$ in $\cA^{rst}$ is still reasonably large, i.e., at least 
$\mu''\,|\cP^{st}|$ for some~$\mu''$ that depends again only on $\mu$.
Finally for some $Z\subseteq Y$ of size $f$ we will manage to pick vertices 
$P^{st}_{\green}$ for $s<t$ from~$Z$ such that whenever $r<s<t$ are from $Z$ the 
triple $P^{rs}_{\red}P^{rt}_{\blue}P^{st}_{\green}$ appears in 
$\cA^{rst}$. For this to succeed we just need $|Y|$ and hence also $|X|$ and $m$ to be 
large enough depending on $f$ and~$\mu$.  We then enumerate $Z=\{\lambda(1), \dots, \lambda(f)\}$
in increasing order to conclude the argument.  

The construction we use for the first stage proceeds in $m^*=|X|$ steps. In the first step
we just select $1\in X$. In the second step we put $2$ into $X$ and we will also make a decision 
concerning~$P^{12}_{\red}$. For that we ask every candidate $k\in [3,m]$ that might be 
put into $X$ in the future to propose suitable choices for $P^{12}_{\red}$. 
This leads us to consider for each such $k$ the set $\cP^{12}_{k,\red}\subseteq \cP^{12}$
of vertices with degree $\tfrac\mu 2\cdot|\cP^{1k}|\,|\cP^{2k}|$ in $\cA^{12k}$. 
Since $\cA$ is $\mu$-dense we have $|\cP^{12}_{k,\red}|\ge \tfrac\mu2\cdot |\cP^{12}|$ for each $k\ge 3$. 
Thus we can choose a vertex $P^{12}_{\red}$ in such a manner that it belongs to 
$\cP^{12}_{k,\red}$ for many $k$'s. From now on we restrict our attention to such 
$k$'s only. The third step begins by putting the smallest such $k$ into $X$. If this happens 
to be, e.g., $7$ then we ask each still relevant $k>7$ for an opinion about 
the possible choices for the pair $(P^{17}_{\red}, P^{27}_{\red})$ and then 
we choose these two vertices in such a way that there are sufficiently many possibilities 
to continue. The general situation after $h$ such steps is described in Lemma~\ref{mh-ac} 
below and the simpler Corollary~\ref{choose-a} contains all that is needed for 
our intended application. 

When reading the statement of the following lemma it might be helpful to think of $M$,~$m$, 
and $\eps$ there as being $m$, $m^*$, and $\tfrac\mu2$ from the outline above. Also, 
$n_1, \dots, n_h$ correspond to the indices which were already put into $X$ whilst 
$n_{h+1}, \dots, n_m$ are the indices that still have a chance of being put into $X$ 
in the future.
 
\begin{lemma}\label{mh-ac}
Given $\eps\in(0, 1)$ and positive integers $m\ge h$ there exists a positive integer 
$M=M(\eps, m,h)$ for which the following is true. 
Suppose that we have 
\begin{enumerate}
\item[$\bullet$] nonempty sets $\cP^{rs}$ for ${1\le r<s\le M}$ and 
\item[$\bullet$] further sets $\cP^{rs}_{t,\red}\subseteq \cP^{rs}$ with 
$|\cP^{rs}_{t,\red}|\ge\eps\,|\cP^{rs}|$ for $1\le r<s<t\le M$, 
\end{enumerate}
then there are indices $n_1<\dots<n_m$ in $[M]$ and there are elements 
$P^{n_rn_s}_{\red}\in \cP^{n_rn_s}$ for ${1\le r<s\le h}$ such that 
\[
	P^{n_rn_s}_{\red}\in \bigcap_{t\in (s,m]}\cP_{n_t,\red}^{n_rn_s}\,.
\]
\end{lemma}
                  
\begin{proof}
We argue by induction on $h$. For the base case $h=1$ we may take ${M(\eps, m, 1)=m}$ 
and $n_r=r$ for all $r\in [m]$; because no vertices $P^{rs}_{\red}$ have to be chosen, 
the conclusion holds vacuously. 

Now suppose that the result is already known for some integer $h$ and all relevant pairs 
of~$\eps$ and~$m$, and that an integer $m\ge h+1$ as well as a real number $\eps\in(0,1)$ 
are given. Set
\[
	m'=h+1+\left\lceil \frac{m-h-1}{\eps^h}\right\rceil 
	\quad \text{ and  } \quad 
	M=M(\eps, m, h+1)=M(\eps, m', h)\,.
\]
To see that $M$ is as desired, let sets $\cP^{rs}$ and $\cP^{rs}_{t,\red}$ as described above 
be given. Due to the definition of $M$, there are indices $n_1<\dots<n_{m'}$ in $[M]$ and 
certain $P^{n_rn_s}_{\red}\in \cP^{n_rn_s}$ such that 
$P^{n_rn_s}_{\red}\in \cP_{n_t,\red}^{n_rn_s}$ holds whenever $1\le r<s<t\le m'$ and $s\le h$. 
We set
\[
	\ccP=\cP^{n_1n_{h+1}}\times\dots\times \cP^{n_hn_{h+1}}\,.
\]
For each $h$-tuple $(P_1, \dots, P_h)\in\ccP$ 
we write
\begin{equation}\label{eq:QP}
	Q(P_1, \dots, P_h)=
	\bigl\{ 
		t\in[h+2, m']
		\colond
		P_r\in \cP_{n_t,\red}^{n_rn_{h+1}} \ \text{for every}\  r\in[h] \bigr\}\,.
\end{equation}
By counting the elements of 
\[
	\{(t, P_1, \dots, P_h)\colond t\in Q(P_1, \dots, P_h)\}
\]
in two different ways and using the lower bounds 
$|\cP_{n_t,\red}^{n_rn_{h+1}}|\ge \eps |\cP^{n_rn_{h+1}}|$ we get
\[
	\sum_{(P_1, \dots, P_h)\in\ccP}|Q(P_1, \dots, P_h)| =
	\sum_{t=h+2}^{m'}\prod_{r=1}^{h}|\cP_{n_t,\red}^{n_rn_{h+1}}| 
	\ge (m'-h-1)\,\eps^h\,|\ccP|\,.
\]
Hence, we may fix an $h$-tuple $(P_1, \dots, P_h)\in\ccP$ with 
\[
	|Q(P_1, \dots, P_h)|\ge (m'-h-1)\eps^h\ge m-h-1\,.
\]
Now let $\ell_{h+2}< \dots< \ell_m$ be any elements from 
\[
	Q=\{n_t\st t\in Q(P_1, \dots, P_h)\}
\]
in increasing order. Set 
\[
	\ell_r=n_r \text{ for all } r\in [h+1]
	\text{ as well as } P^{n_r,n_{h+1}}_{\red}=P_r
	\text{ for all } r\in[h]\,.
\] 

We claim that the indices $\ell_1<\dots <\ell_m$ and the elements $P^{n_rn_s}_{\red}$ 
with $1\le r<s\le h+1$ satisfy the conclusion. To see this let any $1\le r<s<t\le m$ 
with $s\le h+1$ be given. 
We have to verify $P^{\ell_r\ell_s}_{\red}\in \cP_{\ell_t,\red}^{\ell_r\ell_s}$. 
If $s\le h$ this follows directly from $\ell_r=n_r$, $\ell_s=n_s$,  
$\ell_t\in\{n_{s+1}, \dots, n_{m'}\}$, and the inductive choice of the latter set. 
For the case $s=h+1$ if follows from  $t\ge h+2$, that there is some 
$q\in Q(P_1, \dots, P_h)$ with $\ell_t=n_q$. The first property of $q$ entails 
in view of~\eqref{eq:QP} that $P_r\in \cP_{n_q,\red}^{n_rn_{h+1}}$ and, 
as $P^{n_r,n_{h+1}}_{\red}=P_r$, 
this is exactly what we wanted.   
\end{proof}

The reason for having the two parameters $m$ and $h$ in this lemma is just that 
this facilitates the proof by induction on $h$. In applications one may 
always set $h=m$, since this gives the strongest possible conclusion for fixed~$m$. 
Thus it might add to the clarity of exposition if we restate this case again, using the 
occasion to eliminate some double indices as well.

\begin{cor}\label{choose-a}
Suppose that for $M\gg \max(m,\eps^{-1})$ we have 
\begin{enumerate}
\item[$\bullet$] nonempty sets $\cP^{rs}$ for ${1\le r<s\le M}$ and 
\item[$\bullet$] further sets $\cP^{rs}_{t,\red}\subseteq \cP^{rs}$ with 
$|\cP^{rs}_{t,\red}|\ge\eps\,|\cP^{rs}|$ for $1\le r<s<t\le M$, 
\end{enumerate}
then there is a subset $X\subseteq [M]$ of size 
$m$ and there are elements $P^{rs}_{\red}\in \cP^{rs}$ for $r<s$ from~$X$ 
such that 
\[
\pushQED{\qed}
P^{rs}_{\red}\in \bigcap_t\big\{\cP^{rs}_{t,\red}\colond t>s \text{ and } t\in X\bigr\}\,. 
\qedhere
\popQED
\]
\end{cor}      

As discussed above, this statement will be used below for choosing the vertices~$P^{rs}_{\red}$. 
The selection principle we use for choosing the $P^{st}_{\green}$ is essentially the same, 
but we have to apply the symmetry $r\longmapsto M+1-r$ to the indices throughout. 
To prevent confusion when this happens within another argument, we  restate the foregoing result as follows.
      
\begin{cor}\label{choose-c}
Suppose that for $M\gg \max(m,\eps^{-1})$ we have 
\begin{enumerate}
\item[$\bullet$] nonempty sets $\cP^{st}$ for ${1\le s<t\le M}$ and 
\item[$\bullet$] further sets $\cP_{r,\green}^{st}\subseteq \cP^{st}$ with 
$|\cP_{r,\green}^{st}|\ge\eps\,|\cP^{st}|$ for $1\le r<s<t\le M$, 
\end{enumerate}
then there is a subset $Z\subseteq [M]$ of size $m$ and there are elements 
$P^{st}_{\green}\in \cP^{st}$ for $s<t$ from~$Z$ 
such that 
\[
	P^{st}_{\green}\in \bigcap_r\big\{\cP_{r,\green}^{st}\colond r<s \text{ and } r\in Z\bigr\}\,. 
\]
\end{cor}      

\begin{proof}
Set $\cP_*^{rs}=\cP^{M+1-s, M+1-r}$ for $1\le r<s\le M$ and 
$\cP^{rs}_{t,\red}=\cP_{M+1-t,\green}^{M+1-s, M+1-r}$ for $1\le r<s<t\le M$. 
Then apply Corollary~\ref{choose-a}, thus getting a certain set $X$ and some elements 
$P^{rs}_{\red}$. It is straightforward to check that 
\[
	Z=\{M+1-x\colond x\in X\}
\]
and $P^{st}_{\green}=P^{M+1-t, M+1-s}_{\red}$ are as desired.
\end{proof}

The statement that follows coincides with~\cite{RRS-e}*{Lemma~7.1}, where a short direct proof 
is given. For reasons of self-containment, however, we will show here that it follows easily
from the above Corollary~\ref{choose-c}. Subsequently it will be used in the proof of a 
lemma playing a r\^ole similar to that of Lemma~\ref{mh-ac}, but preparing the selection of 
the vertices $P^{rt}_{\blue}$ rather than~$P^{rs}_{\red}$. Specifically, the statement that 
follows will be used in that step of the proof of the next lemma that corresponds to choosing 
$P_1, \dots, P_h$ in the proof of Lemma~\ref{mh-ac}.

\begin{cor}\label{2-indices}
Suppose that for $M\gg\max(m,\eps^{-1})$ we have 
\begin{enumerate}
\item[$\bullet$] nonempty sets $W_1, \dots, W_M$ and
\item[$\bullet$] further sets $D_{rs}\subseteq W_s$ with 
$|D_{rs}|\ge\eps\,|W_s|$ for $1\le r<s\le M$,
\end{enumerate}
then there is a subset $Z\subseteq [M]$ of size $m$ and there are 
elements $d_s\in W_s$ for $s\in Z$ such that 
\[
	d_{s}\in \bigcap_r\big\{D_{rs}\colond r<s \text{ and } r\in Z\bigr\}\,. 
\]
\end{cor}

\begin{proof}
Let $M$ be so large that the conclusion of Corollary~\ref{choose-c} holds with $m+1$ in place 
of~$m$ and with the same $\eps$. Now let the sets $W_s$ and $D_{rs}$ as described above be given.

Set $\cP^{st}=W_s$ for $1\le s<t\le M$ and $\cP_{r,\green}^{st}=D_{rs}$ for $1\le r<s<t\le M$. 
By hypothesis $\cP_{r,\green}^{st}$ is a sufficiently large subset of $\cP^{st}$, so by our 
choice of $M$ there is a set $Z^*\subseteq [M]$ of size $m+1$ together with certain elements 
$P^{st}_{\green}\in \cP^{st}$ for $s<t$ from~$Z^*$ such that 
$P^{st}_{\green}\in \cP_{r,\green}^{st}$ holds whenever $r<s<t$ are from $Z^*$. 
Set $z=\max(Z^*)$, $Z=Z^*\setminus\{z\}$, and $d_s=P^{sz}_{\green}$ for all $s\in Z$. 
We claim that $Z$ and the $d_s$ are as demanded. 

The condition $|Z|=m$ is clear, so now let any pair $r<s$ from $Z$ be given. 
Then $r<s<z$ are from $Z^*$, whence $d_s=P^{sz}_{\green}\in \cP_{r,\green}^{sz}=D_{rs}$. 
\end{proof}

The next lemma deals with the selection of ``blue'' vertices.

\begin{lemma}\label{mh-b}
Given $\eps\in(0, 1)$ and nonnegative integers $m\ge h$ there exists a positive integer 
$M=M(\eps, m,h)$ for which the following is true. 
Suppose that we have 
\begin{enumerate}
\item[$\bullet$] nonempty sets $\cP^{rt}$ for ${1\le r<t\le M}$ and 
\item[$\bullet$] further sets $\cP_{s,\blue}^{rt}\subseteq \cP^{rt}$ 
	with $|\cP_{s,\blue}^{rt}|\ge\eps\,|\cP^{rt}|$ for $1\le r<s<t\le M$,
\end{enumerate}
then there are indices $n_1<\dots<n_m$ in $[M]$ and there are elements 
$P^{n_rn_t}_{\blue}\in \cP^{n_rn_t}$ for all $1\le r<t\le m$ with $r\le h$ such that 
\[
	P^{n_rn_t}_{\blue}\in \bigcap_{s}\big\{\cP_{n_s,\blue}^{n_rn_t}\colond r<s<t\big\}\,.
\] 
\end{lemma}

\begin{proof}
Again we argue by induction on $h$ with the base case $h=0$ being trivial. 

For the induction step we assume that the lemma is already known for some $h$ and 
all possibilities for $m$ and $\eps$, and proceed to the case $m\ge h+1$. 
We contend that $M=M(\eps, m', h)$ is as desired when $m'$ is chosen so large that the 
conclusion of Corollary~\ref{2-indices} holds for $(m'-h-1, m-h-1)$ here in place of 
$(M, m)$ there -- with the same value of $\eps$. 

So let any sets $\cP^{rt}$ and $\cP_{s,\blue}^{rt}$ as described above be given. 
The choice of $M$ guarantees the existence of some indices $n_1<\dots <n_{m'}$ in $[M]$ 
together with certain elements~$P^{n_rn_t}_{\blue}$ satisfying the conclusion of 
Lemma~\ref{mh-b} with $m'$ in place of $m$. The $m$ indices we are requested to find will 
be $n_1, \dots, n_{h+1}$ and $(m-h-1)$ members of the set $\{n_{h+2}, \dots, n_{m'}\}$, 
so in order to gain notational simplicity we may assume $n_r=r$ for all $r\in[m']$. 
Thus we have $P^{rt}_{\blue}\in \cP_{s,\blue}^{rt}$ whenever $1\le r<s<t\le m'$ and 
$r\le h$. 

Let us now define $W_j=\cP^{h+1,h+j+1}$ for all $j\in [m'-h-1]$ and 
$D_{ij}=\cP_{h+i+1,\blue}^{h+1,h+j+1}$ for all $i<j$ from $[m'-h-1]$. 
Then the conditions of Corollary~\ref{2-indices} are satisfied, meaning that there is a 
subset $Z$ of $[m'-h-1]$ of size $m-h-1$ together with certain elements $d_{j}\in W_j$ for 
$j\in Z$ such that we have $d_j\in D_{ij}$ whenever $i<j$ are from $Z$.      

We contend that the set of the $m$ indices we are supposed to find can be taken to be 
\[
	[h+1]\cup\bigl((h+1)+Z\bigr)\,.
\]
To see this we may for simplicity assume $Z=[m-h-1]$, so that the set of our $m$ indices 
is simply $[m]$. Recall that we have already found above certain elements 
$P^{rt}_{\blue}\in \cP^{rt}$ for $1\le r<t\le m$ with $r\le h$ such that 
$P^{rt}_{\blue}\in \cP_{s,\blue}^{rt}$ holds whenever $1\le r<s<t\le m$ and $r\le h$. 
So it remains to find further elements $P^{h+1,t}_{\blue}\in \cP^{h+1,t}$ for $t\in[h+2,m]$ 
with $P^{h+1,t}_{\blue}\in \cP_{s,\blue}^{h+1,t}$ whenever $h+2\le s<t\le m$. To this end, 
we use the vertices obtained by applying Corollary~\ref{2-indices} and
set $P^{h+1, t}_{\blue}=d_{t-h-1}$ for all $t\in[m+2,h]$. Observe that 
$P^{h+1, t}_{\blue}\in W_{t-h-1}=\cP^{h+1, t}$ holds for all relevant $t$. Moreover, if 
$h+2\le s<t\le m$, then we have indeed 
$P^{h+1, t}_{\blue}=d_{t-h-1}\in D_{s-h-1, t-h-1}=\cP_{s,\blue}^{h+1, t}$. 
Thereby the proof by induction on~$h$ is complete.   
\end{proof}

For the same reasons as before we restate the case $h=m$ as follows.

\begin{cor}\label{choose-b}
Suppose that for $M\gg\max(m, \eps^{-1})$ we have 
\begin{enumerate}
\item[$\bullet$] nonempty sets $\cP^{rt}$ for $1\le r<t\le M$ and
\item[$\bullet$] further sets $\cP_{s,\blue}^{rt}\subseteq \cP^{rt}$ with $|\cP_{s,\blue}^{rt}|\ge \eps\,|\cP^{rt}|$
	for $1\le r<s<t\le M$,
\end{enumerate}
then there is a subset $Y\subseteq [M]$ of size $m$ and there are elements 
$P^{rt}_{\blue}\in \cP^{rt}$ for $r<t$ from $Y$ such that
\[
	\pushQED{\qed}
	P^{rt}_{\blue}\in 
	\bigcap_s\big\{\cP_{s,\blue}^{rt}\colond r<s<t \text{ and } s\in Y\bigr\}\,. 
	\qedhere \popQED
\]
\end{cor}

After these preparations we are ready to verify Lemma~\ref{char-modif}.

\begin{proof}[Proof of Lemma~\ref{char-modif}] Suppose 
\[
	m\gg m_*\gg m_{**}\gg \max(f, \mu^{-1})\,.
\]
Consider any three indices $1\le r<s<t\le m$. For a vertex $P\in \cP^{rs}$ we denote the degree 
of $P$ in $\cA^{rst}$ by $d_t(P)$. In other words, this is the number of pairs 
$(Q, R)\in \cP^{rt}\times \cP^{st}$ with $\{P, Q, R\}\in E(\cA^{rst})$. 
Further, we set
\[
	\cP^{rs}_{t,\red}=
	\bigl\{P\in \cP^{rs}\colond d_t(P)\ge\tfrac\mu2\cdot|\cP^{rt}|\,|\cP^{st}|\bigr\}\,.
\]
Since
\begin{align*}
	\mu\,|\cP^{rs}|\,|\cP^{rt}|\,|\cP^{st}| & 
		\le \big|E\bigl(\cA^{rst}\bigr)\big|=\sum_{P\in \cP^{rs}}d_t(P)
	=\sum_{P\in \cP^{rs}\setminus\cP^{rs}_{t,\red}}d_t(P)+\sum_{P\in \cP^{rs}_{t,\red}}d_t(P) \\
		&\le \tfrac\mu2\cdot|\cP^{rs}|\,|\cP^{rt}|\,|\cP^{st}|
			+|\cP^{rs}_{t,\red}|\,|\cP^{rt}|\,|\cP^{st}|\,,
\end{align*}
we have $|\cP^{rs}_{t,\red}|\ge \tfrac\mu2\cdot|\cP^{rs}|$. So applying 
Corollary~\ref{choose-a} with $\bigl( m, m^*, \tfrac\mu2\bigr)$ here in place of 
$(M, m, \eps)$ there we get a set $X\subseteq[m]$ of size $m_*$ together with some 
vertices $P^{rs}_{\red}$ satisfying the condition mentioned there. For simplicity we relabel 
our indices in such a way that $X=[m^*]$, intending to find the required indices 
$\lambda(1), \dots, \lambda(f)$ in $[m^*]$. This completes what has been called the first 
stage of the proof in the outline at the beginning of this section.

Next we look at any three indices $1\le r<s<t\le m_*$. Recall that we just achieved 
$d_t(P^{rs}_{\red})\ge \tfrac\mu2\cdot|\cP^{rt}|\,|\cP^{st}|$. We write $p(P,Q)$ for the 
pair-degree of any two vertices $P\in \cP^{rs}$ and $Q\in \cP^{rt}$ in $\cA^{rst}$, i.e., 
for the number of triples of this hypergraph containing both $P$ and $Q$. 
Let us define
\[
	\cP_{s,\blue}^{rt}
	=\bigl\{Q\in \cP^{rt}\colond p(P^{rs}_{\red},Q)\ge\tfrac\mu4\cdot|\cP^{st}|\bigr\}\,.
\]
Starting from the obvious formula
\[
	d(P^{rs}_{\red})=\sum_{Q\in \cP^{rt}}p(P^{rs}_{\red},Q)\,,
\]
the same calculation as above discloses $|\cP_{s,\blue}^{rt}|\ge \tfrac\mu4\cdot|\cP^{rt}|$. 
So we may apply Corollary~\ref{choose-b} with $\bigl( m_*, m_{**}, \tfrac\mu4\bigr)$ here 
instead of $(M, m, \eps)$ there in order to find a subset $Y$ of $[m_*]$ of size~$m_{**}$ 
together with certain vertices $P^{rt}_{\blue}$. As before it is allowed to suppose 
$Y=[m_{**}]$, in which case we have 
$p(P^{rs}_{\red}, P^{rt}_{\blue})\ge \tfrac\mu4\cdot|\cP^{st}|$ 
whenever $1\le r<s<t\le m_{**}$. 

Having thus completed the second stage we look at any three indices $1\le r<s<t\le m_{**}$. 
Let $\cP_{r,\green}^{st}$ denote the set of all vertices $R$ from $\cP^{st}$ for which the triple 
$\{P^{rs}_{\red}, P^{rt}_{\blue}, R\}$ belongs to $\cA^{rst}$. Due to our previous choices 
we have $|\cP_{r,\green}^{st}|\ge \tfrac\mu4\cdot|\cP^{st}|$. So we may apply 
Corollary~\ref{choose-c} with $\bigl( m_{**}, f, \tfrac\mu4\bigr)$ here rather than 
$(M, m, \eps)$ there, thus getting a certain set $Z\subseteq[m_{**}]$ and certain 
vertices $P^{st}_{\green}\in \cP^{st}$ for $s<t$ from $Z$. As always we may suppose that 
$Z=[f]$, so that $\{P^{rs}_{\red}, P^{rt}_{\blue}, P^{st}_{\green}\}$ becomes a triple of 
$\cA^{rst}$ whenever $1\le r<s<t\le f$. Now it is plain that the indices $\lambda(r)=r$ for 
$r\in [f]$ are as desired. 
\end{proof}

\section{Uniformly dense with vanishing density}
\label{sec:Erdoes}
We reprove Theorem~\ref{thm:ternary} from~\cite{FR88} and we devote to each implication a separate section.
\subsection{The forward implication} The statement that every frequent hypergraph 
is contained in one and, hence, eventually in all sufficiently large ternary hypergraphs,
is a direct consequence of the fact that the sequence $(T_n)_{n\in\NN}$ is itself $d$-dense for an appropriate
function $d\colon (0, 1]\to (0, 1]$. This observation is due to to Erd\H{o}s and S\'os~\cite{ErSo82} who
left the verification to the reader. 
In~\cite{FR88}*{Proposition~3.1} it was shown that the sequence of ternary hypergraphs is $d$-dense 
for some function $d(\eta)=\eta^\rho$ with $\rho>10$. Here we sharpen this estimate and establish Proposition~\ref{prop:toptimal}, which gives the optimal exponent 
\begin{equation}\label{eq:tau}
	\rho=\frac{2}{\log_2(3)-1}\approx3.419\dots\,.
\end{equation}
More precisely, we prove the following lemma, which yields Proposition~\ref{prop:toptimal}.
\begin{lemma}\label{lem:Tnisdense}
For $\rho$ given in~\eqref{eq:tau}, $\l\ge 1$, $X\subseteq V(T_\l)$, and $|X|=\eta\cdot 3^\l$ 
we have
\[
	e(X)\ge\frac{1}{4}\eta^{\rho}\cdot \frac{|X|^3}6-\frac 38\cdot 3^{\l}\,.
\]
\end{lemma}

For the proof of this lemma we shall utilise the following inequality.

\begin{fact}\label{fact:7}
If $x$, $y$, $z\in [0, 1]$ and  $\tau=\rho+3$ for $\rho$ given in~\eqref{eq:tau}, then
\[
	x^\tau+y^\tau+z^\tau+24\,xyz\ge 3^{3-\tau}(x+y+z)^\tau\,.
\]
\end{fact} 

\begin{proof}
In the proof the following identity will be handy to use
\begin{equation}\label{eq:tauid}
	2^{\tau-1}=3^{\tau-3}\,.
\end{equation}
As the unit cube is compact, there is a point $(x_*, y_*, z_*)\in [0, 1]^3$
at which the continuous function $f\colon [0, 1]^3\to\RR$ given by 
\[
	(x, y, z)\longmapsto x^\tau+y^\tau+z^\tau+24\,xyz-3^{3-\tau} (x+y+z)^\tau
\] 
attains its minimum value,
say $\xi$. Due to symmetry we may suppose that $x_*\ge y_*\ge z_*$. Assume 
for the sake of contradiction that $\xi<0$.

Since $\tau>1$, convexity implies 

\[
	x^\tau+y^\tau
	\ge
	2\Big(\frac{x+y}{2}\Big)^\tau 
	=
	2^{1-\tau}(x+y)^\tau
	\overset{\eqref{eq:tauid}}{=}
	3^{3-\tau}(x+y)^\tau\,.
\]
Consequently, $f(x, y, 0)\ge 0$ for all real $x, y\in [0, 1]$ and we have  $x_*, y_*, z_*>0$.

The minimality of $\xi$ implies
\begin{align*}
	x_*^\tau \xi  \le x_*^\tau f\bigl(1, \tfrac{y_*}{x_*}, \tfrac{z_*}{x_*}\bigr)
	&=x_*^\tau+ y_*^\tau+z_*^\tau+24\,x_*^{\tau-3}\cdot x_*y_*z_*-3^{3-\tau}(x_*+y_*+z_*)^\tau \\
	&=\xi+24(x_*^{\tau-3}-1)x_*y_*z_*\,,
\end{align*}
i.e., $24(1-x_*^{\tau-3})x_*y_*z_*\le \xi(1-x_*^\tau)$, which due to the assumption $\xi<0$ is only possible 
if~$x_*=1$. In other words, the function $x\longmapsto f(x, y_*, z_*)$ from $[0, 1]$ to $\RR$ 
attains its minimum at the boundary point $x=1$ and for this reason we have 
$\frac{\rmd f(x, y_*, z_*)}{\rmd x}\big\vert_{x=1}\le 0$, i.e.,
\begin{equation}\label{eq:7a}
	\tau+24\,y_*z_*\le \tau\cdot3^{3-\tau}(1+y_*+z_*)^{\tau-1}\,.
\end{equation}

Next we observe that the function $z\longmapsto f(1, 1, z)$ from $[0, 1]$ to $\RR$
is concave, because 
\begin{align*}
	\frac{\rmd^2f(1, 1, z)}{\rmd z^2}
	&=
	(\tau-1)\tau\left(z^{\tau-2}-3^{3-\tau}(2+z)^{\tau-2}\right)\\
	&=
	(\tau-1)\tau\left(\frac{(3z)^{\tau-2}-3(2+z)^{\tau-2}}{3^{\tau-2}}\right)
	< 
	0\,.
\end{align*}
Together with
\[
	f(1, 1, 0)
	=2-3^{3-\tau}\cdot 2^{\tau}\
	\overset{\eqref{eq:tauid}}{=}0 
	\qqand 
	f(1, 1, 1) = 27-3^{3-\tau}\cdot 3^{\tau}=0
\]
this proves that $f(1, 1, z)\ge 0$ holds 
for all $z\in [0, 1]$, which in view of $x_*=1$ yields $y_*<1$. Thus the function 
$y\longmapsto f(1, y, z_*)$ from $[0, 1]$ to $\RR$ attains its minimum at the interior point 
$y=y_*$ and we infer $\frac{\rmd f(1, y, z_*)}{\rmd y}\big\vert_{y=y_*}= 0$, i.e.,
\[
	\tau y_*^{\tau-1}+24z_*= \tau\cdot3^{3-\tau}(1+y_*+z_*)^{\tau-1}\,.
\]
In combination with~\eqref{eq:7a} this proves $24(1-y_*)z_*\geq\tau(1-y_*^{\tau-1})$
and recalling $y_*\geq z_*$ we arrive at
\begin{equation}\label{eq:7aa}
	24(1-y_*)y_*\ge \tau(1-y_*^{\tau-1})>\frac{32}{5}(1-y_*^5)\,,
\end{equation}
where we used $\tau=\rho+3>6.4$ for the last inequality (see~\eqref{eq:tau}).
Dividing by $(1-y_*)y_*$ leads to
\begin{equation}\label{eq:7b}
	\frac{1+y_*+y_*^2+y_*^3+y_*^4}{y_*}
	=
	\frac{1-y_*^5}{(1-y_*)y_*}	
	\overset{\eqref{eq:7aa}}{<}
	\frac{15}{4}\,.
\end{equation}

Now for the function $h\colon (0, 1)\to\RR$ given by $h(t)=\frac 1t+1+t+t^2+t^3$
we have 
\[
	h'(t)<0 \quad \Longleftrightarrow \quad t^2(1+2t+3t^2)<1\,.
\]
Consequently, there is a unique point $t_*\in (0, 1)$, at which $h$ attains its global minimum
and a short calculation reveals $t_*\in \bigl[\frac{5}{9}, \frac{4}{7}\bigr]$.

From~\eqref{eq:7b} we may now deduce
\[
	\left(\frac 1{t_*}+1+t_*\right)+t_*^2+t_*^3<\frac {15}4\,.
\]

Since $t\longmapsto\frac1t+1+t$ is decreasing on $(0, 1)$, this may be weakened
to
\[
	\frac{7}{4}+1+\frac{4}{7}+\left(\frac{5}{9}\right)^2
	+\left(\frac{5}{9}\right)^3 <\frac{15}4\,,
\]
which, however, is not the case. 
Thus $\xi\ge 0$ and  
Fact~\ref{fact:7} is proved.
\end{proof}

Lemma~\ref{lem:Tnisdense}
follows by a simple inductive argument from the inequality from Fact~\ref{fact:7}.
\begin{proof}[Proof of Lemma~\ref{lem:Tnisdense}]
The case $\l=1$ is clear, since then the right-hand side cannot be positive.
Proceeding inductively we assume from now on that the lemma holds for $\l-1$
in place of $\l$ and look at an arbitrary set $X\subseteq V(T_\l)$.

Let $V(T_\l)=V_1\dcup V_2\dcup V_3$ be a partition of the vertex set of $T_\l$ such that 
\begin{enumerate}
	\item[$\bullet$] each of $V_1$, $V_2$, and $V_3$ induces a copy of $T_{\l-1}$
	\item[$\bullet$] and all triples $v_1v_2v_3$ with $v_i\in V_i$ for $i=1, 2, 3$
		are edges of $T_\l$.
\end{enumerate}

Setting $X_i=X\cap V_i$ and $\eta_i=|X_i|/3^{\l-1}$ for $i=1, 2, 3$ we get
\begin{align*}
	e(X) &=e(X_1)+e(X_2)+e(X_3)+|X_1||X_2||X_3| \\
		 &\ge \left(\frac{\eta_1^{\rho+3}+\eta_2^{\rho+3}+\eta_3^{\rho+3}+24\,\eta_1\eta_2\eta_3}{4}\right)
		 \frac{\bigl(3^{\l-1}\bigr)^3}{6}-3\cdot\frac 38\cdot  3^{\l-1}
\end{align*}
from the induction hypothesis. In view of Fact~\ref{fact:7} it follows that
\begin{equation}\label{eq:7ex}
	e(X)\ge\frac{27\eta^{\rho+3}}{4}\cdot \frac{\bigl(3^{\l-1}\bigr)^3}6-\frac 38\cdot 3^{\l}\,,
\end{equation}
where 
\[
	\eta=\frac{\eta_1+\eta_2+\eta_3}{3}=\frac{|X_1|+|X_2|+|X_3|}{3^\l}=\frac{|X|}{3^\l}\,,
\]
meaning that~\eqref{eq:7ex} simplifies to the desired estimate 
\[
	e(X)\ge\frac{\eta^\rho}{4}\cdot \frac{|X|^3}6-\frac 38\cdot 3^{\l}\,. \qedhere
\]
\end{proof}

We conclude this subsection by observing that frequent hypergraphs on $\l$ vertices must be 
contained in the ternary hypergraph on $3^{\l}$ vertices.

\begin{lemma}\label{lem:decidable}
If a hypergraph $F$ on $\ell$ vertices is frequent, then it is a subhypergraph of the 
ternary hypergraph $T_\ell$.
\end{lemma}

\begin{proof}
It follows from Lemma~\ref{lem:Tnisdense} that there is some $n\in \NN$ with $F\subseteq T_n$.
Thus it suffices to prove that if $F\subseteq T_n$ and $v(F)=\ell$, then $F\subseteq T_\ell$ 
holds as well. We do so by induction on~$\ell$, the base case $\ell\le 3$ being clear. 

Now let any hypergraph $F$ appearing in some ternary hypergraph and with $\ell\ge 4$ 
vertices be given and choose $n\in \NN$ minimal with $F\subseteq T_n$. 
Take a partition $V(T_n)=V_1\dcup V_2\dcup V_3$ 
such that each of $V_1$, $V_2$, and $V_3$ induces a copy of $T_{n-1}$
and such that all further edges of $T_n$ are of the form $v_1v_2v_3$ with $v_i\in V_i$ 
for $i=1, 2, 3$.
By the minimality of $n$ each of the three sets $V_i\cap V(F)$ with $i=1, 2, 3$ contains 
less than $\ell$ vertices, so by the induction hypothesis they induce suphypergraphs
of $T_{n}$ that appear already in $T_{\ell-1}$. Therefore we have indeed 
$F\subseteq T_\ell$. 
\end{proof}

\subsection{The backward implication} For completeness we include a proof of the fact that 
subhypergraphs of ternary hypergraphs are indeed frequent. This proof follows the lines of the work in~\cite{FR88} and will be done by induction on the order of the 
hypergraph whose frequency we wish to establish. In order to carry the induction it
will help us to address the corresponding supersaturation assertion directly. Let us recall
to this end that a {\it homomorphism} from a hypergraph $F$ to another hypergraph $H$ is 
a map $\phi\colon V(F)\longrightarrow V(H)$ sending edges of $F$ to edges of $H$; explicitly,
this means that $\{\phi(x), \phi(y), \phi(z)\}\in E(H)$ is required to hold for every
triple $xyz\in E(F)$. The set of these homomorphisms is denoted by $\Hom(F, H)$
and $\hom(F, H)=|\Hom(F, H)|$ stands for the {\it number} of homomorphisms from $F$ to $H$.

\begin{prop}\label{prop:52}
Given a hypergraph $F$ which is a subhypergraph of some ternary hypergraph and a 
function $d\colon (0, 1)\to (0, 1)$, there are constants $\eta, \xi >0$ such that 
\[
	\hom(F, H)\ge \xi v(H)^{v(F)}
\]
is satisfied by every hypergraph $H$ with the property that $e(U)\ge d(\eps)|U|^3/6$ 
holds whenever $U\subseteq V(H)$, $\eps\in [\eta, 1]$, and $|U|\ge \eps\, |V(H)|$.  
\end{prop}

\begin{proof}
We argue by induction on $v(F)$. The base case $v(F)\le 2$ is clear, since then~$F$ cannot 
have any edge and $\eta=\xi=1$ works. For $v(F)=3$ we take $\eta=1$ as well as~$\xi=d(1)$.
As every edge of $H$ gives rise to six homomorphisms from $F$ to $H$ we get indeed
$\hom(F, H)\ge 6e(H)\ge d(1)v(H)^3$.

For the induction step let a hypergraph $F$ with $v(F)\ge 4$ and a function  
$d\colon (0, 1)\to (0, 1)$ be given. Let $\l\ge 2$ be minimal with~$F\subseteq T_\l$. 
For simplicity we will suppose that $F$ is in fact an induced subhypergraph of $T_\l$. 

Again we take a partition $V(T_\l)=V_1\dcup V_2\dcup V_3$ 
such that $V_i$ spans a copy of $T_{\l-1}$ for $i=1, 2, 3$ and all further edges 
of $T_\l$ are of the form $v_1v_2v_3$ with $v_i\in V_i$ for $i=1, 2, 3$. 
By symmetry we may suppose, after a possible renumbering of indices, that 
$|V(F)\cap V_3|\ge 2$ holds. Let $F_{12}$ and $F_3$ be the restrictions of $F$
to $V_1\cup V_2$ and $V_3$, respectively. Moreover, we will need the hypergraph
$F_*$ arising from $F$ by deleting all but one vertex from $V(F)\cap V_3$. An 
alternative and perhaps helpful description of $F_*$ is that it can be obtained from
$F_{12}$ by adding a new vertex $z$ and all triples $v_1v_2z$ with $v_1\in V(F)\cap V_1$
and $v_2\in V(F)\cap V_2$.   

Intuitively the reason why there should be many homomorphisms from $F$ into an 
$n$-vertex hypergraph $H$ satisfying some local density condition is the following. 
Due to $v(F_*)<v(F)$ we may assume by induction that $\hom(F_*, H)=\Omega(n^{v(F_*)})$.
This means that there is a collection of $\Omega(n^{v(F_{12})})$ homomorphisms $\phi$
from $F_{12}$ to $H$ that can be extended in $\Omega(n)$ many ways to a member of
$\Hom(F_{*}, H)$. For each such $\phi$ the set $A_\phi\subseteq V(H)$ consisting 
of the possible images of the new vertex $z$ in such an extension inherits a local
density condition, because its size is linear, and a further use of the induction hypothesis
shows that there are $\Omega(n^{v(F_3)})$ homomorphisms from $F_3$ to $A_\phi$. 
These homomorphisms can in turn be regarded as extensions of $\phi$ to members of 
$\Hom(F, H)$. This argument can be performed for any $\phi$ and thus we get 
$\Hom(F, H)\ge \Omega(n^{v(F_{12})})\cdot \Omega(n^{v(F_3)})=\Omega(n^{v(F)})$. 

Proceeding now to the details of this derivation let $\eta_*$ and $\xi_*$
denote the constants obtained by applying the induction hypothesis to~$F_*$ and~$d(\cdot)$.
The minimality of $\ell$ implies~$v(F_3)<v(F)$ and therefore we may apply the induction 
hypothesis to $F_3$ and the function ${d'\colon (0, 1)\to (0, 1)}$ defined 
by $\eps\longmapsto d(\eps\cdot\xi_*/2)$, thus obtaining
two further constants~$\eta_3$ and~$\xi_3$. We contend that 
\[
	\eta=\min\bigl(\eta_*, \tfrac 12\xi_*\eta_3\bigr)
	\quad \text{ and } \quad 
	\xi=\frac{\xi_*^{v(F_3)+1}\xi_3}{2^{v(F_3)+1}} 
\]
have the requested properties. 

Now let any hypergraph $H$ with $e(U)\ge d(\eps)|U|^3/6$ for all $\eps\in [\eta, 1]$
$U\subseteq V(H)$  with $|U|\ge \eps\, |V(H)|$ 
be given and put $n=v(H)$. Due to $\eta_*\ge \eta$ we have 
\begin{equation}\label{eq:homF-}
	\hom(F_*, H)\ge \xi_*n^{v(F_*)}\,.
\end{equation}
For every homomorphism $\phi\in \Hom(F_{12}, H)$ we consider the set
\[
	A_\phi=\bigl\{v\in V(H)\colon \phi\cup\{(z, v)\}\in \Hom(F_*, H)\bigr\}
\]
of vertices that can be used for extending $\phi$ to a homomorphism $\phi\cup\{(z, v)\}$ from 
$F_*$ to~$H$. It will be convenient to identify these sets with the subhypergraphs 
of $H$ they induce. Finally we define
\[
	\Phi=\bigl\{\phi\in\Hom(F_{12}, H)\colon |A_\phi|\ge \tfrac 12\xi_*n\bigr\}
\]
to be the set of those homomorphisms from $F_{12}$ to $H$ that admit a substantial number
of such extensions.

Since $v(F_*)=v(F_{12})+1$ we obtain from~\eqref{eq:homF-} 
\[
	\xi_*n^{v(F_{12})+1}\le \sum_{\phi\in\Hom(F_{12}, H)}|A_\phi|
	\le |\Phi|\cdot n + n^{v(F_{12})}\cdot \tfrac 12\xi_*n\,,
\]
whence
\begin{equation}\label{eq:Omega}
	|\Phi| \ge \tfrac 12\xi_* n^{v(F_{12})}\,.
\end{equation}

Moreover it is clear that 
\begin{equation}\label{eq:homF}
	\hom(F, H)=\sum_{\phi\in \Hom(F_{12}, H)}\hom(F_3, A_\phi)
\end{equation}
and the next thing we show is that for every $\phi\in\Phi$ we have
\begin{equation}\label{eq:homF3}
	\hom(F_3, A_\phi)\ge \xi_3\bigl(\tfrac12 \xi_* n\bigr)^{v(F_3)}\,.
\end{equation}
Owing to our inductive choice of $\eta_3$ and $\xi_3$ it suffices 
for the verification of this estimate to show that if $\eps\in[\eta_3, 1]$,
$U\subseteq A_\phi$, and $|U|\ge \eps |A_\phi|$, then $e(U)\ge d'(\eps)|U|^3/6$.
But since $\phi\in \Phi$ leads to $|U|\ge \frac 12 \eps\xi_* n$, 
this follows immediately from $\frac 12 \eps\xi_*\ge \frac 12 \xi_*\eta_3\ge\eta$, 
the definition of~$d'$, and from our choice of $H$. 

Taken together~\eqref{eq:homF},~\eqref{eq:homF3}, and~\eqref{eq:Omega} yield
\[
	\hom(F, H)\ge\sum_{\phi\in \Phi}\hom(F_3, A_\phi)	
	\ge |\Phi| \cdot  \xi_3\bigl(\tfrac12 \xi_* n\bigr)^{v(F_3)}
	\ge \frac{\xi_*^{v(F_3)+1}\xi_3}{2^{v(F_3)+1}} n^{v(F)}\,,
\]
as desired.
\end{proof}

Proposition~\ref{prop:52} implies that all subhypergraphs
of ternary hypergraphs are frequent and combined with Lemma~\ref{lem:decidable}
this shows that 
being frequent is a decidable property.

\section{Concluding remarks} 
\label{sec:conc_rem} 
\subsection{Hypergraphs with uniformly positive density}    
In~\cite{RRS-e}*{Section~2} we defined for a given antichain $\ccA\subseteq \powerset([k])$
and given real numbers $d\in [0, 1]$, $\eta>0$ the concept of a~$k$-uniform hypergraph
being $(d, \eta, \ccA)$-dense. An obvious modification of~\eqref{eq:pi3dot} does then lead 
to corresponding {\it generalised Tur\'{a}n densities} $\pi_{\ccA}(F)$ of $k$-uniform 
hypergraphs~$F$. Now the question presents itself to determine $\pi_{\ccA}(F)$ for all 
antichains $\ccA$ and all hypergraphs~$F$. At the moment this appears to be
a hopelessly difficult task, as it includes, among many further variations, the original 
version of Tur\'{a}n's problem to determine the ordinary Tur\'{a}n density $\pi(F)$ of any 
hypergraph $F$. 

For the time being 
it might be more reasonable to focus on the case $\ccA=[k]^{(k-2)}$ (or stronger density assumptions),
as it might be that for this case one can establish a theory that 
resembles to some extent the classical theory for graphs initiated by Tur\'{a}n himself
and developed further by Erd\H{o}s, Stone, and Simonovits and many others. 

Another possible direction is to characterise for given $\ccA$ the hypergraphs~$F$ 
with ${\pi_{\ccA}(F)=0}$
and here it seems natural to pay particular attention to the \emph{symmetric case}, when $\ccA=[k]^{(j)}$
contains all $j$-element subsets of $[k]$.
Let us now describe an extension of Thereom~\ref{zero} to this setting. 
First of all, a $k$-uniform hypergraph $H=(V, E)$ is said to be {\it $(d, \eta, j)$-dense},
for real numbers $d\in [0, 1]$,~$\eta>0$, and $j\in[k-1]$, if for every $j$-uniform
hypergraph $G$ on~$V$ the collection $\cK_k(G)$ of all $k$-subsets of~$V$ inducing a 
clique~$K^{(j)}_k$ in $G$ obeys the estimate 
\[
	\big|E\cap \cK_k(G)\big|\ge d\,\big|\cK_k(G)\big|-\eta\,|V|^k\,.
\]
One then defines for every $k$-uniform hypergraph $F$ 
\begin{multline*}\label{eq:pi3dot}
	\pi_{j}(F)=\sup\bigl\{d\in[0,1]\colon \text{for every $\eta>0$ and $n\in \NN$ there exists an $F$-free,}\\
	\text{$(d,\eta, j)$-dense, $k$-uniform hypergraph $H$ with $|V(H)|\geq n$}\bigr\}
\end{multline*}
and \cite{RRS-e}*{Proposition 2.5} shows that 
these densities $\pi_{j}(\cdot)$ agree with the densities $\pi_{[k]^{(j)}}(\cdot)$ 
alluded to in the first paragraph of this subsection.

For $j=k-1$ it is known that every $k$-uniform hypergraph $F$ satisfies $\pi_{k-1}(F)=0$, which follows for example 
from the work in~\cite{KRS02}. Thereom~\ref{zero} address the case $j=k-2$ for $k=3$ and for general~$k$ we obtain 
the following characterisation.
\begin{thm}\label{thm:zero-k}
For a $k$-uniform hypergraph $F$, the following are equivalent:
\begin{enumerate}[label=\alabel]
	\item $\pi_{k-2}(F)=0$.
	\item\label{it:zero-k:b} There are an enumeration of the vertex set $V(F)=\{v_1, \dots, v_f\}$ 
		and a $k$-colouring $\phi\colon \partial F\to [k]$ of the $(k-1)$-sets 
		of vertices covered by hyperedges of~$F$ 
		such that every hyperedge $e=\{v_{i(1)},\dots,v_{i(k)}\}\in E(F)$ with 
		$i(1)<\dots <i(k)$ satisfies 
		\begin{equation}\label{eq:phi}
			\phi(e\sm \{v_{i(\ell)}\})=\ell
			\quad \text{ for every } \ell\in [k]\,.
		\end{equation}
\end{enumerate}
\end{thm}
This can be established in the same way as Theorem~\ref{zero}, but using the hypergraph regularity lemma
for $k$-uni\-form hypergraphs. For the corresponding notion of reduced
hypergraphs we refer to~\cite{RRS-e}*{Definition~4.1} and for guidance on the 
reduction corresponding to Section~\ref{sec:reduce} above we refer to the part of the
proof of~\cite{RRS-e}*{Proposition~4.5} presented in Section~4 of that article.

For $j\in[k-3]$ we believe Theorem~\ref{thm:zero-k} extends in the natural way, where the 
$k$-colouring~$\phi$ in part~\ref{it:zero-k:b} is replaced by a $\binom{k}{j+1}$-colouring 
of the $(j+1)$-sets covered by an edge of~$F$ and condition~\eqref{eq:phi} is replaced
by a statement to the effect that the edges of $F$ are rainbow and mutually order-isomorphic 
when one takes these colours into account.   

For $j=0$ such a characterisation leads to
$k$-partite $k$-uniform hypergraphs~$F$ and, hence, such a result renders a common generalisation 
of Erd\H os' result from~\cite{Er64} and Theorem~\ref{thm:zero-k} and
we shall return to this in the near future.

Despite this progress the problem to describe for an arbitrary (asymmetric) antichain $\ccA\subseteq \powerset([k])$
the \hbox{$k$-uniform} class $\{F\colon \pi_{\ccA}(F)=0\}$ remains challenging.
In the $3$-uniform case the investigation of
$\{F\colon \pi_{\ev}(F)=0\}$ and $\{F\colon \pi_{\ee}(F)=0\}$, where $\ev=\{1, 23\}$
and $\ee=\{12, 13\}$, shows that algebraic structures enter the picture and this is 
currently work in progress of the authors.

We close this section with the following questions that compares $\pivvv(F)=\pi_1(F)$ with~$\pi(F)$ for $3$-uniform hypergraphs.
\begin{question}\label{q:pivvvpi}
	Is $\pi_1(F)<\pi(F)$ for every $3$-uniform hypergraph $F$ with $\pi(F)>0$\,? 
\end{question}
Roughly speaking, this questions has an affirmative answer, if no $3$-uniform hypergraph~$F$ with positive Tur\'an density 
has an extremal hypergraph~$H$ that is uniformly dense with respect to large vertex sets $U\subseteq V(H)$
(see also~\cite{ErSo82}*{Problem~7} for a related assertion).
In light of the fact, that all known extremal constructions for such $3$-uniform hypergraphs~$F$ are 
obtained from blow-ups or iterated blow-ups 
of smaller hypergraphs, which fail to be $(d,\eta,1)$-dense for all $d>0$ and sufficiently small $\eta>0$,
the answer to Question~\ref{q:pivvvpi} might be affirmative. Recalling that $\pi(F)=\pi_0(F)$ may suggest 
many generalisations of Question~\ref{q:pivvvpi} to $k$-uniform hypergraphs~$F$ of the form: 
\emph{For which~$F$ do we have $\pi_j(F)<\pi_i(F)$ for $0\leq i<j<k$?} 
At this point this is only known for $i=0$ and $j=k-1$ and Question~\ref{q:pivvvpi} 
is the first interesting open case.

\subsection{Hypergraphs with uniformly vanishing density}
Definition~\ref{dfn:d-dense} admits a straigthforward generalisation to $k$-uniform
hypergraphs: one just replaces all occurrences of the word ``hypergraph'' by 
``$k$-uniform hypergraph'' and all occurrences of the number $3$ by~$k$.

The sequence of ternary hypergraphs generalises to a sequence 
$(T^{(k)}_n)_{n\in \NN}$ of \hbox{$k$-uniform} hypergraphs that might be called {\it $k$-ary} 
and are defined as follows. The vertex set of~$T^{(k)}_n$ is~$[k]^n$ and given
$k$ vertices $\seq{x}_1, \dots, \seq{x}_k$, say with coordinates
$\seq{x}_i=(x_{i1}, \dots, x_{in})$ for $i\in [k]$ one looks at the least number
$m\in [n]$ for which $x_{1m}=\dots=x_{km}$ fails and declares 
$\{\seq{x}_1, \dots, \seq{x}_k\}$ to be an edge of $T^{(k)}_n$ if and only if  
$\{x_{1m}, \dots, x_{km}\}=[k]$ holds. The proof of Theorem~\ref{thm:ternary} (and of Lemma~\ref{lem:decidable})
generalises in the following way (see~\cite{FR88}).

\begin{thm}\label{thm:kvanish}
	A $k$-uniform hypergraph $F$ on $\l$ vertices is frequent if, and only if it is a subhypergraph of
	the $k$-ary hypergraph $T^{(k)}_\l$ on $k^\l$ vertices.\qed
\end{thm}
 
Some further questions concerning frequent hypergraphs arise naturally 
 and below we discuss a few of them.
 
In the context of $3$-uniform hypergraphs one may use three sets instead of one set 
in the definition of $d$-dense (see Definition~\ref{dfn:d-dense}\,\ref{it:d-dense}) 
and this leads to a question that is somewhat different
from the one answered by Theorem~\ref{thm:ternary}. This happens 
because the -- perhaps on first sight expected -- analogue of~\eqref{eq:3vs1} 
does not hold. More explicitly, we say that a sequence $\seq{H}=(H_n)_{n\in\NN}$ of 
$3$-uniform hypergraphs with $v(H_n)\to\infty$ as $n\to\infty$ is 
{\it $(d, \vvv)$-dense} for a function $d\colon(0, 1)\to(0, 1)$
provided that for every $\eta>0$ there is some $n_0\in\NN$ such that
for every $n\ge n_0$ and all choices of $X, Y, Z\subseteq V(H_n)$ 
with $|X||Y||Z|\ge\eta |V(H_n)|^3$ there are at least 
$d(\eta)|X||Y||Z|$ ordered triples $(x, y, z)\in X\times Y\times Z$ 
with $xyz\in E(H_n)$. Besides, a $3$-uniform hypergraph~$F$ is called 
{\it $\vvv$\,-frequent} if for every function~$d\colon(0, 1)\to(0, 1)$ and 
every $(d, \vvv)$-dense sequence $\seq{H}=(H_n)_{n\in\NN}$ of $3$-uniform 
hypergraphs there exists an $n_0\in\NN$ with $F\subseteq H_n$ for every $n\ge n_0$.
			
The relation of this concept to being $d$-dense is as follows: If a 
sequence $\seq{H}$ of $3$-uniform hypergraphs is $(d, \vvv)$-dense,
then, by looking only at the case $X=Y=Z$ in the definition above, one sees 
that $\seq{H}$ is also $d$-dense. On the other hand, being $d$-dense does 
not even imply being $(d', \vvv)$-dense for any function $d'$. As an example
we mention that the sequence of ternary hypergraphs fails to be 
$(d, \vvv)$-dense for every $d\colon(0, 1)\to(0, 1)$.
			
As a corollary of Theorem~\ref{thm:ternary} subhypergraphs of 
ternary hypergraphs are $\vvv$\,-frequent, but the converse implication may not hold.
This leads to the following intriguing problem. 
\begin{problem}
	Characterise $\vvv$\,-frequent $3$-uniform hypergraphs.
\end{problem}

Similar to studying $\pi_j(\cdot)$ for $k$-uniform hypergraphs for every $j<k$ one 
may study dense sequences with respect to different  uniformities. More precisely, for a given integer 
$j\in [k-1]$ and a function~$d\colon(0, 1)\to(0, 1)$
we say that a sequence $\seq{H}=(H_n)_{n\in\NN}$ of $k$-uniform hypergraphs
with $v(H_n)\to\infty$ as $n\to\infty$ is {\it $(d, j)$-dense} if for every 
$\eta>0$ there is an $n_0\in\NN$ such that for every $n\ge n_0$ and 
every $j$-uniform hypergraph $G$ on $V(H_n)$ with $|\cK_k(G)|\ge \eta |V(H_n)|^k$
the estimate 
\[
	\big|E(H_n)\cap \cK_k(G)\big|\ge d(\eta)|\cK_k(G)|
\]
holds. Moreover, a $k$-uniform hypergraph $F$ is defined to be {\it $j$-frequent}
if for every function~$d\colon(0, 1)\to(0, 1)$ and every $(d, j)$-dense sequence
$\seq{H}=(H_n)_{n\in\NN}$ of $k$-uniform hypergraphs there exists an $n_0\in\NN$
with $F\subseteq H_n$ for every $n\ge n_0$. In particular, $1$-frequent is the 
same as frequent in the sense of Theorem~\ref{thm:kvanish}.

Similar as discussed above the $k$-ary hypergraphs show that there is a subtle 
difference between $(d,1)$-dense sequences and $(d,[k]^{(1)})$-dense sequences (where 
we take $k$ sets instead of one set). However, for $j\geq 2$ one can follow the argument 
presented in the proof of~\cite{RRS-e}*{Proposition~2.5} to show that a $k$-uniform 
hypergraph $F$ is $j$-frequent
if and only if it is $[k]^{(j)}$-frequent (defined in the obvious way).
As a result one can show that every $k$-uniform 
hypergraph $F$ is $(k-1)$-frequent by following the inductive proof on the number of edges of the 
counting lemma for hypergraphs. This leaves open 
to characterise the $j$-frequent hypergraphs for $j\in [2, k-2]$. 

Finally, we mention that one may also consider $(d,\ccA)$-dense sequences of 
hypergraphs for asymmetric antichains $\ccA$ and characterising $\ccA$-frequent 
hypergraphs is widely open.

\end{document}